\documentclass{amsart}
\usepackage{amsmath}
\usepackage{hyperref}
\usepackage[utf8]{inputenc}
\usepackage[OT2,T1]{fontenc}
\usepackage{microtype,amsthm}
\numberwithin{equation}{subsection}
\theoremstyle{plain}
\newtheorem{theorem}[subsection]{Theorem}
\newtheorem{proposition}[subsection]{Proposition}
\newtheorem{lemma}[subsection]{Lemma}

\newtheorem{corollary}[subsection]{Corollary}
\theoremstyle{definition}
\newtheorem{definition}[subsection]{Definition}
\newtheorem{instance}[subsection]{Example}
\newtheorem{situation}[subsection]{}

\usepackage[all]{xy}
\author{Dingxin Zhang}
\title{The Monsky--Washnitzer topos}
\hypersetup{
 pdfauthor={Dingxin Zhang},
 pdftitle={The Monsky--Washnitzer topos},
 pdfkeywords={},
 pdfsubject={},
 pdfcreator={Emacs 26.2 (Org mode 9.2.3)},
 pdflang={English}}
\begin{document}

\maketitle

Monsky and Washnitzer defined a good ``\(p\)-adic'' cohomology theory for smooth
affine varieties~\cite{monsky-washnitzer:formal-cohomology-1} over a field
of characteristic \(p > 0\). For arbitrary \(k\)-varieties, Berthelot defined
another cohomology theory known as the rigid cohomology, which agrees with the
cohomology of Monsky--Washnitzer for smooth affine \(k\)-varieties. The original
definition of rigid cohomology was not site-theoretic. For a proper
\(k\)-variety \(X\), Ogus defined a ``convergent
site''~\cite{ogus:convergent-topos}, and the cohomology of a certain sheaf on
this site agrees with the rigid cohomology of \(X\). For arbitrary \(X\),
Le~Stum had defined an ``overconvergent
site''~\cite{lestum:overconvergent}, and cohomology of a sheaf on this site
computes the rigid cohomology of \(X\).

In this paper, we define, for any
\(k\)-variety \(X\), another site, the \emph{Monsky--Washnitzer site} of
\(X\). We prove the cohomology of a sheaf on this site also computes the
a certain analytic cohomology of \(X\). According to Große-Klönne~\cite[Theorem~5.1(c)]{grosse-klonne:rigid-analytic-spaces-with-overconvergent-structure-sheaf},
these analytic cohomology groups agree with rigid cohomology groups.

Comparing with Le~Stum's overconvergent site,
we feel the Monsky--Washnitzer site is more related to the construction of
Monsky--Washnitzer (whereas Le~Stum's site is more related to Berthelot's
original theory).

\medskip{}
The idea is to mimic the every construction made by Ogus in his convergent
theory~\cite{ogus:convergent-topos} (also Shiho's logarithmic
version of Ogus's theory~\cite{shiho:crystalline-fundamental-group-2}), with
formal schemes replaced by weak formal schemes defined by
Meredith~\cite{meredith:weak-formal-scheme}.

According to~\cite[\S2.3]{grothendieck:crystals-de-rham-cohomology},
it seems such a
theory has been perceived by Monsky and Washnitzer. Although Grothendieck
commented that the method of Monsky--Washnitzer is ``too closely bound to
differential forms, which practically limits its applications to smooth
schemes'', the difficulty is overcome by embedding the scheme in question into a
smooth one. the method of differential forms can then be extended to possibly
singular schemes, just like Grothendieck and Berthelot did to crystalline
cohomology, or Ogus did for convergent cohomology.

Most of results in the paper have parallel versions in the convergent topos, and
the proofs of these do not contain surprise. One exception is that the ``Theorem
B'' for quasi-Stein dagger spaces cannot be deduced from the original version by
the standard argument (it is a much harder problem). In the special situation we
need, such a theorem is made available by
Bambozzi~\cite{bambozzi:theorem-b-dagger-quasi-stein}.

The main result of this note is that the cohomology of the Monsky--Washnitzer
topos of a certain natural sheaf agrees with an ``analytic cohomology'', which in
turn agrees with rigid cohomology of Berthelot, see
Corollary~\ref{corollary:comparison}. It is also possible to treat twisted
coefficients, but we do not pursue this further.

\medskip{}\noindent
\textbf{Notation.}
Throughout this paper we fix the following notation.

\begin{itemize}
\item Let \(k\) be a field of characteristic \(p > 0\).
\item Let \(V\) be a complete discrete valuation ring of characteristic \(0\) with a uniformizer \(\varpi\) such that \(V/\varpi V = k\).
\item Let \(K\) be the fraction field of \(V\), equipped with the standard absolute
value \(|\cdot|\) such that \(|p| = p^{-1}\).
\item Let \(K^{\text{a}}\) be an algebraic closure of \(K\).
\item For an extension \(L\) of \(K\) with an absolute
value \(|\cdot|\) extending the one on \(K\), we use \(\mathcal{O}_{L}\) to
denote the valuation ring of \(L\), i.e., \(\{x \in L : |x| \leq 1\}\). In
particular \(V = \mathcal{O}_K\).
\item For each \(m\)-tuple of positive numbers
\(\lambda = (\lambda_1,\ldots,\lambda_m)\),
\(\lambda_i\in |K^\times| \otimes \mathbb{Q}\), we set
\(|\lambda| = \max \lambda_i\), and
\begin{equation*}
\mathfrak{T}_m(\lambda) = \left\{\sum_{\substack{I = {i_1,\ldots,i_m} \\ i_1, \ldots i_m \geq 0}}
a_I t_1^{i_1}\cdots t_m^{i_m} \in K [\![t_1,\ldots,t_m]\!]:
|a_{I}| \lambda_i^{i_1} \cdots \lambda_n^{i_m} \to 0 \right\}.
\end{equation*}
Then \(\mathfrak{T}_m(1,\ldots,1)\), which will simply be denoted by
\(\mathfrak{T}_m\), is just the usual Tate algebra
\(K\langle t_1,\ldots,t_m\rangle\).
\item If \(X\) is a site, \(\textit{Shv}(X)\) or \(X^{\sim}\) are used to denote the
sheaf topos of \(X\).
\end{itemize}

\medskip{}\noindent
\textbf{Acknowledgment.} I thank Shizhang Li, who taught me rigid analytic geometry in his unique
and refreshing way, and answered many of my questions when I was preparing the
manuscript. I also thank him for his constant persuasion that spurred me to
finish the project. I am grateful to B.~Lian, T.-J.~Lee, A.~Huang and
S.-T.~Yau, and Yang~Zhou.

\section{Weak formal schemes}
\label{sec:orgcaff186}

We shall review the notion of weakly complete, finitely generated
algebras à la Monsky--Washnitzer and the notion of weak formal schemes à la
Meredith.
\begin{situation}
\label{situation:affinoid-vs-dagger-algebras}
Let \(K\langle t_1,\ldots, t_n \rangle\) be the Tate algebra of restricted power
series of \(n\)-variables with coefficients in \(K\):
\begin{equation*}
K\langle t_1,\ldots, t_n \rangle = \{f \in K [\![t]\!] : |f(a_1,\ldots,a_n)| < \infty,
\forall (a_1,\ldots,a_n) \in \mathcal{O}_{K^{\text{a}}}^{n}\}.
\end{equation*}
Thus \(K\langle t_1,\ldots, t_n\rangle\) is the ring of ``analytic functions''
defined over \(K\), that are convergent on the unit polydisk in
\(K^{\text{a}}\). We equip \(K\langle t_1,\ldots, t_n \rangle\) the \(p\)-adic
topology. We use \(V\langle t_1,\ldots, t_n\rangle\) to denote the ring of power
bounded elements (which are precisely power series with coefficients in \(V\))
in \(K\langle t_1,\ldots, t_n \rangle\).
An \emph{affinoid} algebra is a homomorphic image of a Tate algebra.
Basic properties of affnoid algebras are recorded nicely
in~\cite{bosch:formal-rigid}; in our exposition we assume some familiarity
with these notions.

We say \(f \in K\langle t_1, \ldots, t_n \rangle\) is \emph{overconvergent} if there
exists a number \(\rho > 1\) (depending on \(f\)) such that
\begin{equation*}
|f(a_1,\ldots,a_n)| < \infty, \quad \forall (a_1,\ldots,a_n) \in K^n,  |a_i| \leq \rho.
\end{equation*}
Clearly all overconvergent elements in \(K\langle t_1,\ldots, t_n\rangle\) form
dense subring of \(K\langle t_1,\ldots, t_n \rangle\), which is called the
\emph{Monsky--Washnitzer algebra}, and is denoted by
\(K\langle t_1,\ldots,t_n \rangle^{\dagger}\). We also set
\begin{equation*}
V\langle t_1,\ldots,t_n\rangle^{\dagger} = K\langle t_1,\ldots,t_n \rangle^{\dagger}\cap V\langle t_1,\ldots,t_n \rangle
\end{equation*}
and call it the \emph{integral Monsky--Washnitzer algebra}, and is usually for
convenience denoted by \(W_{m}(K)\) or just \(W_m\).
A \emph{dagger algebra} is a homomorphic image of a Monsky--Washnitzer algebra. \label{org6d45d32}

Dagger algebras admit two different topologies, the topology inherited
from the residual norm of a presentation (which is called the \emph{affinoid topology} in the sequel),
and a colimit topology. The latter is defined as follows: write \(A = W_m/I\), where
\(I\) is generated by elements in \(\mathfrak{T}_m(\rho_0)\) for some \(\rho_0 > 1\),
then we can write \(A=\mathrm{colim}_{1<\rho<\rho_0}\mathfrak{T}_m(\rho)/I\mathfrak{T}_m(\rho)\).
Thereby \(A\) receives a colimit topology (which is called the ``fringe topology''
by K.~S.~Kedlaya) from this presentation. Both the
affinoid topology and the colimit topology of \(A\) is independent of the choice
of the presentation. We will mostly use the affinoid topology, henceforth
referred to as \emph{the} topology; when the colimit topology is used we will explicitly
specify.

Let \(A\) be a dagger algebra. Then its affinoid topology is not complete (i.e., there
exists Cauchy sequence that is not convergent). The \emph{completion} \(\widehat{A}\)
of \(A\) can be characterized as follows:
\end{situation}

\begin{lemma}
\label{lemma:completion-vs-presentation}
If \(A = W_m/(f_1,\ldots,f_r)\) be a (presented) dagger algebra.
Then the completion of \(A\) is given by
\(\widehat{A} = K\langle t_1,\ldots,t_m\rangle / (f_1,\ldots,f_r)\).
\end{lemma}

\begin{proof}
The algebra \(K\langle t_1,\ldots,t_m\rangle / (f_1,\ldots,f_r)\), being a
quotient of the Tate algebra by an ideal (necessarily
closed~\cite[(2.3/8)]{bosch:formal-rigid}), is a complete Banach algebra.
The assertion then follows from the fact that the natural map
\(A \to K\langle t_1,\ldots,t_m\rangle / (f_1,\ldots,f_r)\) is injective and has
dense image.
\end{proof}

We would like to record the universal property of the Monsky--Washnitzer algebra.
Recall that an element \(a \in A\) in a topological \(k\)-algebra is \emph{power bounded} if for
any neighborhood \(U\) of \(1_A\), there exists a neighborhood \(V\) of \(1_A\),
such that \(V \cdot \{a^n: n\in \mathbb{N}\}\) is contained in \(U\). An element
\(a\in A\) is \emph{weakly power bounded} if \(\lambda a\) is power bounded for some
\(\lambda \in K\), \(|\lambda| > 1\). The set of
power bounded elements is denoted by \(A^{\circ}\). The set of weakly power
bounded elements in \(A\) is denoted by \(A^{\diamond}\). Let \(A\) be a dagger
algebra with completion \(\widehat{A}\), then it is not hard to see that
\(A^{\circ} = \widehat{A}^{\diamond}\).

If \(A\) is an affinoid algebra over \(K\), then \(A^{\circ}\) is the subring of
\(A\) consisting of elements whose supremum seminorm is less than or equal to
\(1\).

\begin{lemma}
\label{lemma:universal-property-rational-mw-algebra}
Let \(R\) be a dagger \(K\)-algebra.
Let \(r_1,\ldots,r_m\) be a collection of power bounded elements in \(R\).
Then there exists a unique continuous homomorphism
\(\mathrm{ev}_{r_1,\ldots,r_m}: K\langle t_1,\ldots,t_m\rangle^{\dagger}\to R\)
such that \(t_i \mapsto r_i\). Thus, in the category of dagger algebras, the
functor \(R \mapsto (R^{\circ})^{n}\) is represented by \(W_m\).
\end{lemma}

\begin{proof}
Let \(A\) be a dagger \(K\)-algebra.
Let \(0 \to I \to W_m \to A \to 0\) be a presentation of \(A\) by a
Monsky--Washnitzer algebra.
For each \(\lambda \in (|K^{\times}| \otimes \mathbb{Q})^{m}\),
\(|\lambda| > 1\), there is an inclusion \(\mathfrak{T}_{m}(\lambda) \to W_m\).
Denote by \(A(\lambda)\) the image of \(\mathfrak{T}_{m}(\lambda)\) in \(A\). Then
\(A = \bigcup_{|\lambda|>1} A(\lambda)\).

Let \(a \in A\) be a power bounded element. Then \(a\) induces a ring
homomorphism \(\varphi: \mathfrak{T}_1 = K\langle t \rangle \to \widehat{A}\)
sending \(t\) to \(a\), where \(\widehat{A} = \mathfrak{T}_m/I\mathfrak{T}_m\) is the affinoid algebra
attached to \(A\). For each \(\rho \in |K^{\times}| \otimes \mathbb{Q}\),
\(\rho > 1\), we must prove that \(\varphi(\mathfrak{T}_1(\rho))\) is contained in
\(A(\lambda)\) for some suitable \(\lambda\).

Let \(f(t_1,\ldots,t_m)\) be a power bounded preimage of \(a\) in
\(W_m \subset \mathfrak{T}_m\). Then there exists \(\lambda'\) such that
\(f(t_1,\ldots,t_n) \in \mathfrak{T}_m(\lambda')\), and is power bounded there. This means
that the supremum norm of \(f(t_1,\ldots,t_n)\) on the polydisk of radius
\(\lambda'\) is \(\leq 1\).

By our choice, the map \(\mathfrak{T}_1 \to \widehat{A}\) factors as
\begin{equation*}
\mathfrak{T}_1 \xrightarrow{t \mapsto f} \mathfrak{T}_m \to \widehat{A}.
\end{equation*}
Thus it suffices to prove the map \(\mathfrak{T}_1 \to \mathfrak{T}_m\) sends \(\mathfrak{T}_1(\rho)\) into
some \(\mathfrak{T}_m(\lambda)\).

Let \(g \in \mathfrak{T}_1(\rho)\) be arbitrary.
We must prove that \(g(f(t_1,\ldots,t_m))\) is
convergent on some closed disk of radius \(|\lambda| > 1\).
Since the \(\lambda'\)-Gauß norm of \(f(t_1,\ldots,t_m)\) is a continuous
function in \(\lambda'\), and since the 1-Gauß norm of \(f\) is at most \(1\)
(this is the condition that \(f\) is power bounded in \(W_m\)), it follows that
there exists \(\lambda\) such that the \(\lambda\)-Gauß norm of \(f\) is at most
\(\rho\). It follows that \(g(f(t_1,\ldots,t_m))\) falls in
\(\mathfrak{T}_m(\lambda)\): indeed, as supremum norm of \(\mathfrak{T}_m(\lambda)\) is the same as
the \(\lambda\)-Gauß norm on \(\mathfrak{T}_{m}(\lambda)\), for each point
\((c_1,\ldots,c_m) \in \mathrm{Sp}(\mathfrak{T}_m(\lambda))\), we have
\(C = |f(c_1,\ldots,c_m)| < \rho\). Thus
\(|g(f(c_1,\ldots,c_m))| = |\sum b_j C^j| < \infty\). This proves the lemma.
\end{proof}
\begin{situation}
\label{situation:weakly-complete}
\textbf{Weakly complete algebras.}
Let \(R\) be topological \(V\)-algebra with \(\varpi\)-adic topology. Recall
that \(R\) is \emph{complete} if \(R = \lim_n R/\varpi^{n}R\) (in particular, we \emph{do}
require \(R\) to be \(\varpi\)-adically separated).
In this case, for any
\(r_1, \ldots, r_n \in R\) and any \(f \in V\langle t_1,\ldots,t_n \rangle\),
we can always ``evaluate'' \(f\) at \((r_1,\ldots, r_n)\), and get back an element
\(f(r_1,\ldots, r_n) \in R\).
We say that (the \(\varpi\)-adic topology of) \(R\) is \emph{weakly complete}, if it is
\(\varpi\)-adically separated (therefore the canonical map \(R \to \widehat{R}\)
is injective), and for any finite collection of elements \(r_1,\ldots,r_n \in
R\), any \emph{overconvergent} power series
\(f \in V\langle t_1,\ldots,t_n \rangle^{\dagger}\) with coefficients in \(V\),
the element \(f(r_1,\ldots,r_n) \in \widehat{R}\) falls in \(R\).

For example, any \(\varpi\)-adically complete \(V\)-algebra is weakly complete.
The integral Monsky--Washnitzer algebra \(V\langle t_1,\ldots,t_n \rangle^{\dagger}\) is
weakly complete. The polynomial algebra \(V[t_1,\ldots,t_n]\) is not weakly
complete.

The \emph{weak completion} of a \(\varpi\)-adic \(V\)-algebra \(R\) is the smallest
subalgebra \(R^{\dagger}\) of its \(\varpi\)-adic completion \(\widehat{R}\)
that is weakly complete. For example, the weak completion of
\(V[t_1,\ldots,t_n]\) is \(V\langle t_1,\ldots,t_n \rangle^{\dagger}\). If \(R\)
is already weakly complete, then \(R = R^{\dagger}\) by fiat.

We say a subset \(S\) of a \(\varpi\)-adic \(V\)-algebra \emph{weakly generates} \(R\),
if for any \(r \in R\), there exist finitely many elements
\(s_1,\ldots, s_m \in S\) (\(m\) depends on \(r\))
and an element \(f \in V\langle t_1,\ldots, t_m\rangle^{\dagger}\), such that
\(r = f(s_1,\ldots,s_m)\). We say \(R\) is a \emph{weakly complete, finitely generated}
\(V\)-algebra if \(R\) is weakly generated by a finite subset. Hence, weakly
complete, finitely generated algebras are precisely the homomorphic images of
some integral Monsky--Washnitzer algebras. Therefore, if \(R\) is a weakly
complete, finitely generated algebra over \(V\), \(R[1/\varpi]\) is a \hyperref[org6d45d32]{dagger
\(K\)-algebra}.

To keep synchronized with the world of formal schemes, we shall say a weakly
complete, finitely generated algebra \(R\) over \(V\) is \emph{admissible} if it is
flat over \(V\). This is equivalent to saying \(R\) has no \(\varpi\)-torsion
elements.
\end{situation}

\begin{situation}
\label{situation:weak-completion-tensor-product}
\textbf{Weak completed tensor product.}
The category of weakly complete, finitely generated algebras admits tensor
products. Let \(A_1\) and \(A_2\) be two weakly complete, finitely generated
algebras over a weakly complete, finitely generated algebra \(B\) over \(V\).
Then we define \(A_1 \otimes_{B}^{w} A_2\) to be the smallest weakly
complete subalgebra  of
\(\widehat{A}_1\widehat{\otimes}_{\widehat{B}}\widehat{A}_2\) containing
the image of \(A_1 \otimes_B A_2\).
\end{situation}

\begin{lemma}
\label{lemma:completion-commutes-with-inverting}
Let \(R\) be a weakly complete, finitely generated algebra over
\(V\). Then the completion of the dagger algebra \(R[1/\varpi]\) is
\(\widehat{R}[1/\varpi]\).
\end{lemma}

\begin{proof}
Let \(R = W_m^{\circ}/(f_1,\ldots,f_r)\) be a presentation of \(R\). Then the
\(\varpi\)-adic completion of \(R\) is
\(V\langle t_1,\ldots,t_m\rangle/(f_1,\ldots,f_r)\), since both algebras have the same
quotient modulo \(\varpi^{N}\). Then the present lemma follows from
Lemma~\ref{lemma:completion-vs-presentation}.
\end{proof}

\begin{lemma}
\label{lemma:finite-quotient-mw-algebra}
Let \(R\) be an admissible weakly complete, finitely generated algebra over \(V\). Then the
set of maximal ideals of \(R[1/\varpi]\) is in bijective correspondence with
\(R\)-algebras \(V'\) such that
\begin{enumerate}
\item \(R \to V'\) is surjecitve,
\item the composition \(V \to R \to V'\) is an integral extension, and
\item \(V'\) has no \(\varpi\)-torsion elements.
\end{enumerate}
\end{lemma}

In the sequel, a quotient of \(R\) or \(\widehat{R}\) of the form \(V'\) in the
statement will be called a \(V\)-\emph{rig-point} of \(R\) or \(\widehat{R}\).

\begin{proof}
By
\cite[Theorem~1.7(2)]{grosse-klonne:rigid-analytic-spaces-with-overconvergent-structure-sheaf},
the maximal ideals of \(R[1/\varpi]\) are in one-to-one correspondence with the
maximal ideals of its completion. By
Lemma~\ref{lemma:completion-commutes-with-inverting}, the completion is
\(\widehat{R}[1/\varpi]\). Since \(R\) is admissible, so is \(\widehat{R}\),
since the latter is faithfully flat over \(R\)
\cite[Theorem~1.7(1)]{grosse-klonne:rigid-analytic-spaces-with-overconvergent-structure-sheaf}.
Since the points of the rigid analytic space
\(\mathrm{Sp}(\widehat{R}[1/\varpi])\) are in bijective correspondence with
\(V\)-rig-point of \(\widehat{R}\)
\cite[(8.3/3), (8.3/6)]{bosch:formal-rigid},
it suffices to prove that \(V\)-rig-point of \(R\) and \(\widehat{R}\)
agree. First, as any finite \(V\)-algebra is automatically \(p\)-adically complete,
a homomorphism \(R \to V'\) naturally factors through \(\widehat{R}\). So it
suffices to prove that for any surjective map
\(\widehat{\varphi}: \widehat{R}\to V'\), the composition
\(\varphi: R \to \widehat{R} \to V'\) remains surjective. Lifting to a
presentation, it suffices to assume \(R = W_m\). Let \(c_i \in V'\) be the image
of \(t_i\). Since \(V'\) is finite over \(V\), any restricted power series of
\(c_i\) turns to be a polynomial of \(c_i\). This means that \(c_i\) generate
\(V'\) not only topologically, but also algebraically. Hence the composition
\begin{equation*}
V[t_1,\ldots,t_m] \to V\langle t_1,\ldots,t_m\rangle \to V'
\end{equation*}
is surjective. This completes the proof.
\end{proof}
Next, we recall the notion of weak formal schemes of
Meredith~\cite{meredith:weak-formal-scheme}.

\begin{situation}
\label{situation:weak-formal-affine}
Let \(R\) be a weakly complete, finitely generated algebra. In this paragraph
we, following Meredith, define a locally topologically ringed space, which is
denoted by \(\mathrm{Spwf}(R)\).

The ambient set of \(\mathrm{Spwf}(R)\) is the set of open prime ideals of
\(R\). Since \(R\) has \(\varpi\)-adic topology, an ideal is open if and only if
it contains some power of \(\varpi\). This implies that the ambient set of \(R\)
agrees with \(\mathrm{Spec}(R\otimes_V k)\). We give \(\mathrm{Spwf}(R)\) the
induced Zariski topology.

Let \(f \in R\), define the \emph{dagger localization} of \(R\) at \(f\) to be
\begin{equation*}
R\left\langle \frac{1}{f} \right\rangle{}^{\dagger} = R\langle t \rangle^{\dagger}/(tf-1).
\end{equation*}
Plainly, the ambient set of \(\mathrm{Spwf}(R\langle 1/f\rangle^{\dagger})\) is
the same as \(\mathrm{Spec}(R[1/f] \otimes_V k)\).
For a finitely generated \(R\)-module \(M\), define \(M\langle 1/f\rangle^{\dagger}\) to
be \(M \otimes_R R\langle 1/f\rangle^{\dagger}\).
\end{situation}

\begin{theorem}[Meredith]
\label{theorem:dagger-localization-define-sheaf-on-basis}
Let \(R\) be a weakly complete, finitely generated algebra over \(V\).
Let \(M\) be a finitely generated \(R\)-module.
\begin{enumerate}
\item For any nonzero \(f \in R\), \(R\langle 1/f \rangle^{\dagger}\) is flat over
\(R\).
\item The presheaf
\(\mathrm{Spec}(R[1/f]\otimes_V k) \mapsto M\langle 1/f\rangle^{\dagger}\)
on the category of principal open affine subschemes of
\(\mathrm{Spec}(R\otimes_{V}k)\) is a sheaf-on-a-basis, thus defines a sheaf
on \(\mathrm{Spwf}(R)\), which is denoted by \(\widetilde{M}\).
\item \(\mathrm{H}^{i}(\mathrm{Spwf}(R),\widetilde{M})=0\) for all \(i > 0\).
\end{enumerate}
\end{theorem}

\begin{proof}
See~\cite[Theorem 14]{meredith:weak-formal-scheme}.
\end{proof}

\begin{definition}
\label{definition:weak-formal-scheme}
A \emph{weak formal scheme} over \(V\) is a locally topologically ringed space which is
locally isomorphic to \((\mathrm{Spwf}(R), \widetilde{R})\) for some weakly
complete, finitely generated algebra \(R\) over \(V\).
\end{definition}

\begin{situation}
\label{situation:completion-weak-formal-scheme}
\textbf{Weak completion.}
(i) Let \(X\) be a separated \(V\)-scheme of finite type. Then we can also
define the \emph{weak completion} \(X^{\dagger}\) of \(X\) along its special fiber
\(X \otimes_V k\). This will be a weak formal scheme obtained by gluing the weak
completion of open affine schemes of \(X\).
For example,
\(\mathbb{A}^{n,\dagger}_V=\mathrm{Spwf}(V\langle{}t_1,\ldots,t_n\rangle^{\dagger})\).

(ii) Let \(\mathfrak{X}\) be a weak formal scheme over \(V\). Using the completion
functor \(R \mapsto \widehat{R}\), and gluing, we can always define the
completion \(\widehat{\mathfrak X}\) of a weak formal scheme \(\mathfrak{X}\),
which is a formal scheme over \(V\). There is a canonical morphism of locally
and topologically ringed spaces \(\widehat{\mathfrak X} \to \mathfrak{X}\).

(iii) Let \(X\) be a separated finite type \(V\)-scheme, with weak completion
\(X^{\dagger}\). Then the completion of \(X\) along its special fiber is the
same as the completion \(\widehat{X^{\dagger}}\) of the weak formal scheme
\(X^{\dagger}\).
\end{situation}

\section{Dagger spaces and generic fibers of weak formal schemes}
\label{sec:orgd31decc}
In this section we discuss how to associate a generic fiber to a weak formal
scheme. We then define the specialization functor which can be used to define a
crucial notion in this paper, the tubular neighborhood. We then adapt
Berthelot's weak fibration theorem in the present context.

\begin{situation}
\label{situation:dagger-spaces}
\textbf{Dagger spaces.}
Just like formal schemes have rigid analytic spaces as generic fibers, one can
define generic fibers for weak formal schemes. These geometric gadgets are known
as \emph{dagger spaces}. The foundation of dagger spaces appeared
in~\cite{grosse-klonne:rigid-analytic-spaces-with-overconvergent-structure-sheaf}.
Just like rigid analytic spaces, dagger spaces are not genuine ringed spaces,
but spaces with Grothendieck topology. The local pieces of a dagger space are
dagger algebras, just like affinoid algebras are local pieces of rigid analytic
spaces. We shall generally refer the reader to Große-Klönne's paper for basic
properties of dagger spaces.

Introducing dagger spaces into the story is needed in computing the ``analytic
cohomology'' of a \(k\)-variety. The procedure is as follows. Start with a
\(k\)-variety \(X\) that embeds into the special fiber a smooth weak formal
scheme \(P\) over \(V\). Then as we shall define shortly, we can take the ``tube''
\((X)_P\) of \(X\) in \(P\), which will be a dagger space over \(K\). The
analytic cohomology of \(X\) is then the de~Rham cohomology of the dagger
space \((X)_P\). This depends on the choice of the ``frame'' \(P\), but as we
shall show in \S\ref{sec:org0ac363b}, the analytic
cohomology agrees with the sheaf cohomology of the Monsky--Washnitzer
site.
In~\cite[Theorem~5.1(c)]{grosse-klonne:rigid-analytic-spaces-with-overconvergent-structure-sheaf},
it is shown that the analytic cohomology defined above agrees with Berthelot's
rigid cohomology. Combining these results we shall get our main result.
\end{situation}

\begin{situation}
\label{situation:completion-dagger-space}
\textbf{Associated rigid analytic spaces.}
Let \(X\) be a dagger space.
According to~\cite[Theorem~2.19]{grosse-klonne:rigid-analytic-spaces-with-overconvergent-structure-sheaf},
there exists a rigid analytic space \(\widehat{X}\),
called the \emph{completion} of \(X\) or \emph{associated rigid analytic space} of \(X\),
together with a morphism of locally G-ringed spaces \(\widehat{X} \to X\).
This morphism is the terminal object in the category of morphisms of
G-ringed spaces \(Y \to X\) where \(Y\) is a rigid analytic space. The
local construction is as follows: if \(X = \mathrm{Sp}(A)\),
then \(\widehat{X} = \mathrm{Sp}(\widehat{A})\), where \(\widehat{A}\) is the completion
of topological ring \(A\)
(cf.~Lemma~\ref{lemma:completion-vs-presentation}).
\end{situation}

\begin{instance}
\label{example:weierstrass-domain}
Let \(X = \mathrm{Sp}(A)\) be an affinoid dagger space. We say an affinoid
subspace \(Y\) of \(X\) is a Weierstrass domain, if \(\widehat{Y}\) is a
Weierstraß domain of \(\widehat{X}\). Hence, Weierstraß subdomains of \(A\) are
affinoid  subdomains defined by the dagger algebras of the form
\begin{equation*}
B = A\langle t_1,\ldots,t_n \rangle^{\dagger}/(t_i - f_i:i=1,2,\ldots n)
\end{equation*}
for some \(f_i \in A\). Here are two observations about Weierstraß domains.
\begin{enumerate}
\item The ambient set of \(\mathrm{Sp}(B)\) is \(\{x \in X : |f_i(x)|\leq 1 \}\).
This is because passing to completion does not change the ambient set and the
rigid analytic space \(\mathrm{Sp}(\widehat{B})\) has the said form.
\item \emph{The map \(A \to B\) has dense image}. This is because the quotient of the
smaller ring \(A[t_1,\ldots,t_n]/(t_i-f_i)\) is dense in the completion
\(\widehat{B}\).
\end{enumerate}
\end{instance}

\begin{situation}
\label{situation:rigid-space-attached-to-formal-scheme}
\textbf{Generic fiber of weak formal scheme.}
Any weak formal \(V\)-scheme \(\mathfrak{X}\) admits a dagger space
\(\mathfrak{X}_K\) as its ``generic fiber''. The construction is essentially the
same as the Raynaud generic fiber of a formal scheme. We give the construction
in steps (cf.~\cite[\S0.2]{berthelot:rigid-cohomology-compact-support}).
\end{situation}

(1) If \(\mathfrak{X}=\mathrm{Spwf}(A)\) is an affine weak formal scheme, then
\(A[1/\varpi] = A\otimes K\) is a dagger algebra
(see~\ref{situation:weakly-complete}). In this case, \(\mathfrak{X}_K\) is
just the affinoid dagger space \(\mathrm{Sp}(A \otimes K)\) (whose ambient set
is the set of maximal ideals of \(A \otimes K\). As we have argued in
Lemma~\ref{lemma:finite-quotient-mw-algebra}, the points of
\(\mathfrak{X}_K\) [which are the same as the points of
\(\widehat{\mathfrak{X}}_{K}\)
(\cite[Theorem~1.7(2)]{grosse-klonne:rigid-analytic-spaces-with-overconvergent-structure-sheaf})]
are in a bijection with the quotients of \(A\) that are integral, finite, flat
over \(V\), i.e., \emph{rig-points} of \(A\).

If \(V'\) is a rig-point, then \(V'\otimes K\) is a finite extension of
\(K\), defining a maximal ideal of \(\widehat{A} \otimes K\) which  in turn
determines a unique maximal of \(A \otimes K\). Conversely, if \(K'\) is a
finite extension of \(K\) defined by a maximal ideal of \(A \otimes K\), the
image \(R\) of \(A\) in \(K'\) is an integral flat \(V\)-algebra, i.e., a rig
point of \(A\) (since \(A\) and \(\widehat{A}\) have the same image on \(K'\),
by Lemma~\ref{lemma:finite-quotient-mw-algebra}).

If \(t_1,\ldots,t_n\) are weak generators of the \(V\)-algebra \(A\),
then their images in \(K'\) contained in the valuation ring of \(K'\), and are
consequently integral over \(V\). The ring \(V'\) generated by them is thus
finite over \(V\). Since \(V\) is henselian, and \(V'\) is an integral domain
and finite, it follows that \(V'\) is a local \(V\)-algebra, defining a
formal subscheme \(\mathrm{Spwf}(V') \subset \mathrm{Spwf}(A)\) supported at a
single closed point of \(\mathfrak{X}\), called the \emph{specialization} of the point
\(x\in\mathrm{Sp}(A\otimes{}K)\) corresponding to \(V'\).

\medskip
(2) Now suppose that \(\mathfrak{X}\) is a weak formal \(V\)-scheme.
We define \(\mathfrak{X}_K\), as a set, to be the set
of all closed formal subschemes \(Z\) of \(\mathfrak{X}\) that are integral,
finite, flat over \(V\). The support of such a subscheme \(Z\) is a closed point
of \(\mathfrak{X}\), which will be called the \emph{specialization} of the point
\(x\in\mathfrak{X}_K\) corresponding to \(Z\). By associating to any point
\(x\in\mathfrak{X}_K\) its specialization, we get a set-theoretic map
\begin{equation*}
\mathrm{sp}: \mathfrak{X}_K \longrightarrow \mathfrak{X}.
\end{equation*}
For all open affine \(U = \mathrm{Spwf}(A) \subset \mathfrak{X}\),
\(\mathrm{sp}^{-1}(U)\) is in bijection with \(\mathrm{Sp}(A\otimes{}K)=U_K\),
which has a structure of dagger space. The dagger structure of
\(\mathfrak{X}_K\) is furnished by the following proposition.

\begin{proposition}[Cf.~\cite{berthelot:rigid-cohomology-compact-support}, Proposition (0.2.3)]
\label{proposition:defining-raynaud-generic-fiber}
Let \(\mathfrak{X}\) be a weak formal \(V\)-scheme.
Then there is a unique structure of dagger space on
the set \(\mathfrak{X}_K\) defined
in~(\ref{situation:rigid-space-attached-to-formal-scheme}) above, such that
the following conditions hold.
\begin{enumerate}
\item The inverse image of an open subscheme under the map
\(\mathrm{sp}: \mathfrak{X}_K \to \mathfrak{X}\) is an open subspace of
\(\mathfrak{X}_K\).
\item The inverse image of an open covering under \(\mathrm{sp}\) is an admissible
covering for \(\mathfrak{X}_K\).
\item For all open affine \(U \subset X\), the structure induced by
\(\mathfrak{X}_K\) on \(\mathrm{sp}^{-1}(U)\) agree with \(U_K\) defined
in~(\ref{situation:rigid-space-attached-to-formal-scheme}) above.
\end{enumerate}
Moreover, the set-theoretic map \(\mathrm{sp}\) induces a morphism of ringed
topoi
\begin{equation*}
\mathrm{sp}: (\textit{Shv}(\mathfrak{X}_K),\mathcal{O}_{\mathfrak{X}_K})
\to (\mathfrak{X}_{\text{Zar}}^{\sim},\mathcal{O}_{\mathfrak{X}}).
\end{equation*}
The dagger space \(\mathfrak{X}_K\) is called the \emph{generic fiber}
of \(\mathfrak{X}\).
\end{proposition}

\begin{proof}
The uniqueness is clear. For the existence, we choose an open covering of
\(\mathfrak{X}\) by affine weak formal schemes \(U_i = \mathrm{Spwf}(A_i)\). For
each \(f \in A_i\) there is an open affine \(D(f) \subset U_i\). Regard \(f\) as an
element in \(A_i \otimes K = \Gamma(U_{iK},\mathcal{O}_{U_{iK}})\), we then have,
set-theoretically
\begin{equation}
\label{eq:inverse-specialization-of-principal-open}
\mathrm{sp}^{-1}(D(f)) = \{x \in U_{iK} : |f(x)| \geq 1\}.
\end{equation}
In fact, if \(x \in U_{iK}\) corresponds to the quotient \(R\) of \(A_i\), then
the point \(\mathrm{sp}(x)\) is in \(D(f)\) if and only if \(f\) is \emph{not} in the
maximal ideal of \(R\), which means that \(f \in A \otimes K\) is not in the
maximal ideal of the valuation ring \(V(x) \subset K(x)\). Therefore, we infer
that \(x \in U_{iK}\) if and only if \(|f(x)| \geq 1\). On the other hand,
Moreover, we have
\begin{equation*}
\Gamma(D(f),\mathcal{O}_{\mathfrak{X}}) = A_i\langle{}T\rangle^{\dagger}/(fT - 1),
\end{equation*}
hence \(D(f)_K\) is a dagger space associated to
\((A_i\otimes{}K)\langle{}T\rangle^{\dagger}/(fT-1)\), which indeed underlies
the set \(\mathrm{sp}^{-1}(D(f))\).

Since the ambient topological spaces of \(U_i\) are noetherian, the open
\(U_i\cap U_j\) is quasi-compact, hence is a finite union of open subsets of the
form \(D(f)\), with \(f \in A_i\) (resp.~\(f \in A_j\)); it follows that
\begin{equation*}
U_{iK} \cap U_{jK} = \textstyle\bigcup_{f} D(f)_K
\end{equation*}
is open in \(U_{iK}\) (resp.~\(U_{jK}\)). Fix a finite open cover of
\(U_i\cap U_j\) by the opens \(D(f_\alpha)\), \(f_{\alpha} \in A_i\), and, for
all \(\alpha\), a finite open covering \(D(f_{\alpha\beta})\) of \(D(f_\alpha)\)
with \(f_{\alpha\beta} \in A_j\). Then \(\mathrm{sp}^{-1}(D(f_{\alpha\beta}))\)
form a finite open covering of \(U_{iK} \cap U_{jK}\) by the special domains of
\(U_{iK}\) and \(D(f_\alpha)_K\); but as \(D(f_\alpha)_K\) is a special domain
in \(U_{jK}\), \(\mathrm{sp}^{-1}(D(f_{\alpha\beta}))\) are also special domains
of \(U_{iK}\). Therefore, \(\mathrm{sp}^{-1}(D(f_{\alpha\beta}))\) form an
admissible covering of \(U_{iK} \cap U_{jK}\) in \(U_{iK}\) and \(U_{jK}\). It
follows that the structures of rigid analytic space on the domains
\(D(f_{\alpha\beta})\) induced by \(U_{iK}\) and \(U_{jK}\) are the same. We can
thus define a rigid analytic space structure on \(\mathfrak{X}_K\).

Recall that saturated subset of a set \(A\) with respect to a map \(f: A \to B\)
are subsets of \(A\) of the form \(f^{-1}(C)\) for \(C \subset B\). In this
terminology, the open subsets described above defining \(\mathfrak{X}_K\) are
saturated with respect to the specialization map
\(\mathrm{sp}:\mathfrak{X}_K\to\mathfrak{X}\). Therefore to check the condition
(1) and (2) we can work locally on \(\mathfrak{X}\), hence we can assume that
\(\mathfrak{X} = \mathrm{Spwf}(A)\). Since the open subsets \(D(f)\) form a basis
of the topology of \(\mathrm{Spwf}(A)\), and since \(\mathrm{sp}^{-1}(D(f))\) are
indeed admissible open subsets of \(\mathrm{Sp}(A)\). This proves (1). (2) is
proven similarly. The condition (3) is automatic by our construction.

Finally, by (1) and (2), \(\mathrm{sp}\) induces a morphism of categories with
pretopologies, hence a morphism of sheaf topoi
\(\mathrm{sp}:\textit{Shv}(\mathfrak{X}_{K})\to\mathfrak{X}_{\text{Zar}}^{\sim}\).
If \(U\) is an open affine of \(\mathfrak{X}\), we have by construction
\begin{equation*}
  \Gamma(U,\mathrm{sp}_\ast \mathcal{O}_{\mathfrak{X}_K})
= \Gamma(\mathrm{sp}^{-1}U,\mathcal{O}_{\mathfrak{X}_K})
= \Gamma(U,\mathcal{O}_{\mathfrak{X}}) \otimes K.
\end{equation*}
The presheaf \(V \mapsto \Gamma(V,\mathcal{O}_{\mathfrak{X}}) \otimes K\) is a
sheaf on \(U\), because the subspace \(U\) is noetherian. We therefore get an
identification
\begin{equation*}
\mathrm{sp}_\ast \mathcal{O}_{\mathfrak{X}_K} = \mathcal{O}_{\mathfrak{X}} \otimes K,
\end{equation*}
and, for all open subset \(U\), the homomorphism
\begin{equation*}
\Gamma(U,\mathcal{O}_{\mathfrak{X}}) \otimes K
\to \Gamma(U,\mathrm{sp}_\ast\mathcal{O}_{\mathfrak{X}_K})
\end{equation*}
is an isomorphism whenever \(U\) is quasi-compact (hence noetherian). In
particular, the morphism \(\mathrm{sp}\) is a morphism of ringed topoi.
\end{proof}

\begin{situation}
\label{situation:tubular-neighborhood}
Let \(\mathcal{I}\) be a coherent open ideal of a weak formal scheme
\(\mathfrak{X}\). Let \(Z\) be the vanishing scheme of \(\mathcal{I}\). The \emph{open
tubular neighborhood} of \(Z\) in \(\mathfrak{X}\), notation
\((Z)_{\mathfrak{X}}\), is the subset of \(\mathfrak{X}_K\) defined by
\(\mathrm{sp}^{-1}(Z)\). (As a side comment, we shall not need tubular
neighborhoods for locally closed immersions.)
A priori the tubular neighborhood is only a subset of \(\mathfrak{X}_{K}\), but
we claim \((Z)_{\mathfrak{X}}\) \emph{is an admissible open subspace of}
\(\mathfrak{X}_{K}\). The problem being local, so we can assume \(\mathfrak{X}\)
is affine. Then the next lemma then justifies our assertion.
\end{situation}

\begin{lemma}
\label{lemma:equation-and-tube}
Let the notation be as in~\ref{situation:tubular-neighborhood}.
Assume that \(\mathfrak{X} = \mathrm{Spwf}(R)\) is affine,
\(\Gamma(\mathfrak{X},\mathcal{I})\) is generated by \(f_1,\ldots,f_r \in R\).
Then
\begin{align*}
  (Z)_{\mathfrak{X}} &= \{x \in \mathfrak{X}_K: |f_i(x)|<1 \} \\
                     &= \{x \in \mathfrak{X}_K : |f(x)| < 1, \forall f \in I\}.
\end{align*}
is an admissible open subspace of \(\mathfrak{X}_{K}\) which is a nested union
of Weierstraß domains.
\end{lemma}

\begin{proof}
Since
\(Z = \mathfrak{X}\setminus\bigcup D(f_i)=\mathfrak{X}\setminus\bigcup_{f\in I} D(f)\),
we have the set-theoretic equality
\(\mathrm{sp}^{-1}(Z)=\bigcap\mathfrak{X}_{K}\setminus\mathrm{sp}^{-1}(D(f_i))=\bigcap_{f\in{}I}\mathfrak{X}_{K}\setminus\mathrm{sp}^{-1}(D(f))\).
Then both equality immediately follows
from~\eqref{eq:inverse-specialization-of-principal-open}.
Clearly
\begin{equation}
\label{eq:tube-union-weierstrass}
\{x \in \mathfrak{X}_K: |f_i(x)|<1 \} = \bigcup_{n} \{x \in \mathfrak{X}_{K}: |f_i(x)| \leq \varepsilon_n\}
\end{equation}
where \(\varepsilon_n \in |K^{\times}| \otimes \mathbb{Q}\) is a sequence of
numbers converging to \(1\). Each subspace
\begin{equation*}
\{x \in \mathfrak{X}_{K}: |f_i(x)| \leq \varepsilon_n\}
\end{equation*}
is a Weierstraß domain of \(\mathrm{Sp}(R\otimes K)\). This justifies the second
assertion.
\end{proof}

\begin{lemma}
\label{lemma:completing-tubes-get-tubes}
Let the notation be as in \ref{situation:tubular-neighborhood}. The completion (see
\ref{situation:completion-dagger-space}) of the dagger space
\((Z)_{\mathfrak{X}}\) equals the rigid analytic tubular neighborhood of \(Z\)
in \(\widehat{\mathfrak{X}}\).
\end{lemma}

\begin{proof}
We can assume \(\mathfrak{X}\) is affinoid. In the view
of~\eqref{eq:tube-union-weierstrass},
we further reduce ourselves to
proving that the dagger Weierstraß domains defined by \(|f_i| \leq \varepsilon\)
have the rigid analytic Weierstraß domains defined by \(|f_i| \leq \varepsilon\)
as their completions. This follows from
Lemma~\ref{lemma:completion-vs-presentation}.
\end{proof}

\begin{definition}
\label{definition:quasi-stein}
A \emph{quasi-Stein} dagger space \(U\) is a dagger space that can be written as a
nested admissible union of affinoid subdomains \(U_i\), each \(U_i\) is a
Weierstraß domain of \(U_{i+1}\). Let \(f: U' \to U\) be a morphism of dagger
spaces. We say that \(f\) is a \emph{quasi-Stein morphism}, if there exists an
admissible open covering of \(U\) by affinoid subdomains \(V_i\) such that
\(f^{-1}(V_i)\) are quasi-Stein spaces, i.e., spaces that are nested union of
Weierstraß domains.
\end{definition}

\begin{corollary}
\label{corollary:quasi-stein-immersion}
Let the notation be as in~\ref{situation:tubular-neighborhood}.
The morphism \((Z)_{\mathfrak{X}} \to \mathfrak{X}_{K}\) is a quasi-Stein
morphism.
\end{corollary}

\begin{instance}
\label{example:affine-space-product-tube}
Let \(R\) be a weakly complete, finitely generated algebra. Let
\(I = (f_1,\ldots,f_r)\) be an open ideal of \(R\). Then the tubular
neighborhood of \(Z = \mathrm{Spec}(R/I)\) in
\(\mathfrak{X}=\mathbb{A}^{n,\dagger}_{R}\) is the
subspace in \(\mathrm{Sp}(R\otimes K\langle t_1,\ldots, t_s\rangle^{\dagger})\)
defined by \(|f_i| \leq 1\) and \(|t_j| < 1\), that is, the product
\((Z)_{\mathfrak{X}} \times D(0;1^{-})^{n}\). Here, \(D(0;1^{-})\) is the rigid
analytic unit \emph{open} disk. As the open unit disk is partially proper, there is no
need to specify whether we are considering the rigid analytic space or the
dagger space, as the two notion are equivalent for partially proper
spaces~\cite[Theorem~2.27]{grosse-klonne:rigid-analytic-spaces-with-overconvergent-structure-sheaf}.
\end{instance}

Next we prove Berthelot's weak fibration theorem in the context of weak formal
scheme following Berthelot's strategy.

\begin{lemma}
\label{lemma:etaleinvarinace-tubular}
Assume there exists a commutative diagram
\begin{equation*}
\xymatrix{
  & \mathfrak{X} \ar[dd]^{f} \\
  Z \ar[ru]^{u} \ar[rd]_{v} & \\
  & \mathfrak{Y}
}
\end{equation*}
in which \(f\) is a morphism of weak formal schemes, \(u\) and \(v\) are closed
embeddings of \(Z\) into the special fibers of \(\mathfrak{X}\) and
\(\mathfrak{Y}\) respectively. Assume that the completion \(\widehat{f}\) of
\(f\) is an étale morphism. Then
\(f\) induces an isomorphism of tubular neighborhoods
\begin{equation*}
(Z)_{\mathfrak{X}} \xrightarrow{\sim} (Z)_{\mathfrak{Y}}
\end{equation*}
\end{lemma}

\begin{proof}
By Lemma~\ref{lemma:completing-tubes-get-tubes}, the completion of the
tubular neighborhoods \((Z)_{\mathfrak{X}}\) and \((Z)_{\mathfrak{Y}}\) are the
rigid analytic tubular neighborhoods \((Z)_{\widehat{\mathfrak{X}}}\) and
\((Z)_{\widehat{\mathfrak{Y}}}\).
By~\cite[Proposition~1.3.1]{berthelot:rigid-cohomology-compact-support},
these rigid analytic tubular neighborhood are isomorphic. The lemma then follows
from~\cite[Theorem~2.19(4)]{grosse-klonne:rigid-analytic-spaces-with-overconvergent-structure-sheaf},
which asserts that a morphism of dagger spaces is an isomorphism if and only if
its completion is.
\end{proof}

\begin{lemma}
\label{lemma:weak-fibration-theorem}
Assume there exists a commutative diagram
\begin{equation*}
\xymatrix{
  & \mathfrak{X} \ar[dd]^{f} \\
  Z \ar[ru]^{u} \ar[rd]_{v} & \\
  & \mathfrak{Y}
}
\end{equation*}
in which \(f\) is a morphism of admissible weak formal schemes, \(u\) and \(v\) are closed
embeddings of \(Z\) into the special fibers of \(\mathfrak{X}\) and
\(\mathfrak{Y}\) respectively. Assume that the completion \(\widehat{f}\) of
\(f\) is a smooth morphism of formal schemes
relative dimension \(n\). Then
there exists an admissible open covering \(\{V_{\alpha}\}\) of
\((Z)_{\mathfrak{Y}}\) such that, if \(U_{\alpha}\) is the preimage of
\(V_{\alpha}\) under the natural morphism
\((Z)_{\mathfrak{X}} \to (Z)_{\mathfrak{Y}}\), then there is an
\(V_{\alpha}\)-isomorphism
\begin{equation*}
U_{\alpha} \xrightarrow{\sim} V_{\alpha} \times D(0;1^{-})^n
\end{equation*}
of dagger spaces over \(V_{\alpha}\).
Cf.~\cite[Théorème~1.3.2]{berthelot:rigid-cohomology-compact-support}.
\end{lemma}

\begin{proof}
The problem is Zariski local on \(Z\), so we are free to shrink.
On the level of formal schemes, one can find a factorization of the morphism
\(\widehat{f}\)
\begin{equation*}
\mathfrak{X} \xrightarrow{h} \widehat{\mathbb{A}}^{n}_V \times \mathfrak{Y} \xrightarrow{\mathrm{pr}_2} \mathfrak{Y},
\end{equation*}
where \(h\) is an étale morphism of formal schemes, and \(Z\) embeds in the zero
section of \(\widehat{\mathbb{A}}^{n}_V \times \mathfrak{Y}\). The second arrow
descends to weak formal schemes, but \(h\) does not necessarily descend.
However, by the Artin approximation theorem for weakly complete, finitely generated
algebras, see \cite[(2.4.2)]{van-der-put:monsky-washnitzer-cohomology}, we can
find an approximation \(g\) of \(h\), which agrees with \(h\) on special fibers,
and such that \(g\) descends to weak formal
schemes. If we could show \(g\) is étale, then we can apply
Lemma~\ref{lemma:etaleinvarinace-tubular} above and conclude the proof. Note
that \(g\) reduces to an étale morphism on the special fiber, so in order to
prove \(g\) itself is étale, we need prove that it is flat. This can be
deduced from a suitable version of local criterion of flatness. Let
me collect all needed information. By Fulton's theorem, weakly complete,
finitely generated algebras are Noetherian, hence they satisfy the Artin-Rees
lemma, i.e., of type (APf) defined
in~\cite[7.4.11 above]{fujiwara-kata:foundation-rigid-geometry}. Since
the completions of the affine weak formal schemes are formal spectra of complete
\(V\)-algebras, they are of \(\varpi\)-adic Zariski type,
see~\cite[Proposition 7.3.5]{fujiwara-kata:foundation-rigid-geometry}.
In particular, if
\(\mathfrak{X} = \mathrm{Spwf}(A)\), \(\mathfrak{Y} = \mathrm{Spwf}(B)\),
then \(\widehat{A}\) is
\(\varpi(\widehat{B}\langle t_1,\ldots t_n\rangle) \cdot \widehat{A}\)-Zariski,
i.e., \(\varpi{}\widehat{A}\)-Zariski.
Taking \(I = (\varpi)\),
we wish to apply~\cite[Proposition~8.3.8]{fujiwara-kata:foundation-rigid-geometry},
condition (c), so we need to verify that
\(\mathrm{Tor}^{\widehat{B}\langle{}t_1,\ldots,t_n\rangle}_{1}(\widehat{A},B_0[x_1,\ldots,x_n])\)
is zero, where \(B_0 = \widehat{B}/\varpi\). Since we have assumed our
weak formal schemes are admissible, \(\widehat{A}\) has no \(\varpi\)-torsion.
Thereby tensoring the exact sequence
\begin{equation*}
0 \to \widehat{B}\langle t_1,\ldots, t_n \rangle \xrightarrow{\cdot \varpi} \widehat{B}\langle t_1, \ldots, t_n \rangle \to B_0[t_1,\ldots,t_n]\to 0
\end{equation*}
with \(\widehat{A}\) preserves exactness. This checks the vanishing of the first
Tor group. Thereby we have verified all needed conditions to ensure the
flatness. This completes the proof.
\end{proof}

\section{The Monsky--Washnitzer site}
\label{sec:org925345b}

In this section we define the Monsky--Washnitzer site for a variety \(X\) over
\(k\) (relative to \(V\)). We also prove some basic functoriality of the topos
of this site that are needed in later sections.
\begin{definition}
\label{definition:widening}
Let \(X\) be a \(k\)-variety. A \emph{widening} for \(X\) is a pair \((P,z:Z \to X)\)
(which is sometimes referred to as \(P\) if no confusion is likely), where \(P\)
is an admissible weak formal scheme, \(Z\) a subscheme of \(P\) defined by a
coherent open ideal, and \(z\) a morphism of \(k\)-schemes. We say \(P\) is
\emph{affine} if \(z\) is an affine morphism; we say \(P\) is absolutely affine if
\(P\) is affine. To paraphrase, an \emph{absolutely affine widening} \(T\) of \(X\)
consists of
\begin{itemize}
\item a morphism of \(k\)-schemes \(\mathrm{Spec}(R) \to X\),
\item a weakly complete, finitely generated, flat \(V\)-algebra \(A\), and
\item a surjective ring homomorphism \(A \to R\).
\end{itemize}
A \emph{morphism} of widenings
\(u: (P_1, z_1:Z_1\to X) \to (P_2, z_2:Z_2 \to X)\)
is a commutative diagram
\begin{equation*}
\xymatrix{
 P_1  \ar[d]_{u}  & Z_1 \ar[l] \ar[d]^{u_0} \ar[r]^{z_1} & X \ar@{=}[d]\\
 P_2      & Z_2 \ar[l]\ar[r]_{z_2} & X
}.
\end{equation*}
When a widening is absolutely affine, we also use the ring-theoretic notation
\(A \to R \leftarrow \mathcal{O}_X\) to denote it, instead of the geometric notation.

An \emph{enlargement} is a widening \((P,Z,z)\) such that
\(Z\) agrees with \(P\otimes_V k\). Thus an absolutely
affine widening \(T = (A \to R \leftarrow \mathcal{O}_X)\) is an \emph{absolutely
affine enlargement} of \(X\) if \(R = A/\varpi{}A\).
\end{definition}

\begin{situation}
\label{situation:remark-terminology}
\textbf{Remark on terminologies.}
In Ogus's paper~\cite{ogus:convergent-topos}, a \emph{widening} is a pair
\((P,z:Z\to X)\), where \(P\) is a flat formal scheme over \(V\) with \(Z\) as a
subscheme of definition (i.e., \(P\) is \(\mathcal{I}_{Z}\)-adically complete).
A pair \((P,z:Z \to X)\) with \(P\) admissible and \(Z \subset P\otimes_V k\) is
known as a \emph{prewidening} (following
Shiho~\cite{shiho:crystalline-fundamental-group-2}). In the convergent
topos, a prewidening and its corresponding widening represent the same sheaf.
From this perspective, our widening seems to be better called a prewidening, as
weak formal schemes are always \(\varpi\)-adic. But it does not seem to
be very useful to further define weak formal schemes that are not
\(\varpi\)-adic, so we shall stick with the current terminology.
\end{situation}

One attempts to use widenings to define a site. But there is no reasonable
meaning of fiber products of widenings. This is caused by the presence of
\(\varpi\)-torsions in the fiber product of weakly complete, finitely generated
algebras.

\begin{instance}[No obvious fiber product for widenings]
\label{example:no-obvious-fiber-product-for-widening}
Consider the following diagram of widenings of \(X = \mathrm{Spec}(k)\).
(For the ease of thinking, I also put the ``rigid analytic'' counter part of the
picture on the left.)
\begin{equation*}
\xymatrix{
K\langle s \rangle^{\dagger}   & V \langle s \rangle^{\dagger} \ar[l] \ar[r]^{s,\varpi\mapsto 0} & k \ar@{=}[d] \\
K\langle t,s\rangle^{\dagger} \ar[u]^{t\mapsto \varpi} \ar[d]_{t\mapsto 0}  & V\langle t,s \rangle^{\dagger}
\ar[l] \ar[r]_{\varpi\mapsto 0}^{\quad t,s\mapsto 0}\ar[u]^{t\mapsto \varpi} \ar[d]_{t\mapsto 0} & k \\
K\langle s \rangle^{\dagger}  & V \langle s \rangle^{\dagger} \ar[l] \ar[r]_{s, \varpi\mapsto 0}& k \ar@{=}[u]
}
\end{equation*}
In this case, the completed tensor product of the left column is zero.
Geometrically, the two homomorphisms correspond to \(s \mapsto (s,\varpi)\)
and \(s \mapsto (s,0)\), two disjoint embeddings of the closed disk into the
2-dimensional polydisk. However, the fiber product of the special fiber is
nonzero. Note that the fiber product of the middle column equals \(V[s]/\varpi\),
which is \(\varpi\)-torsion. Trying to flatten it will result the zero ring.
\end{instance}

\begin{situation}
\label{situation:absolute-product-widening}
\textbf{Absolute product exists.}
Nevertheless, the notion of absolute product of widenings can be defined as
follows. If \(T_1 = (A_1,R_1)\) and \(T_2 = (A_2,R_2)\), then we define
\(T_1 \times T_2 = (A_1 \otimes^{w}_V A_2, R_1 \otimes_R R_2)\), where
\(\mathrm{Spec}(R)\) is an open affine scheme on which the morphisms
\begin{equation*}
\mathrm{Spec}(R_i) \to X
\end{equation*}
factor into.
\end{situation}
\begin{lemma}
\label{lemma:flatness-tensor-product}
Let \(R\) be a finite type \(k\)-algebra. Assume that we are given
\begin{itemize}
\item flat, \(\varpi\)-adic, \(\varpi\)-adically complete \(V\)-algebras \(A\),
\(A_1\) and \(A_2\), such that \(A/\varpi{}A = A_i/\varpi{}A_i = R\);
\item ring homomorphisms \(\varphi_i : A \to A_i\), such that \(\varphi_i\) modulo
\(\varpi\) agrees with the identity homomorphism \(\mathrm{Id}: R \to R\).
\end{itemize}
Then \(A_1 \widehat{\otimes}_A A_2\) is flat over \(V\).
\end{lemma}

\begin{proof}
Set \(A_i^{(n)} = A_i/\varpi^{n}A_i\), and similarly set
\(A^{(n)} = A/\varpi^{n}A\). Then the completed tensor product is the inverse
limit of \(B_n = A_1^{(n)} \otimes_{A^{(n)}} A_2^{(n)}\). Any \(\varpi\)-power
torsion of the completed tensor product \(B\) would then give rise elements in
\(b_n \in B_n\), for \(n\)-sufficiently large, that is killed by \(\varpi^{r}\).
But \(\varpi^{r} \in V/\varpi^{n}V\), and \(b_n\) becomes a torsion element. So
\(B_n\) would not be flat. Therefore, it thus suffices to prove
\(B_n = A_1^{(n)} \otimes_{A^{(n)}} A_2^{(n)}\) is flat over \(V/\varpi^{n}\). But
this follows from a suitable version of local criterion of flatness
\cite[p132, Lemma 5.21(3)]{fantechi-et-al:fga-explained}.
\end{proof}

\begin{corollary}
\label{corollary:fiber-product-enlargement}
Let \(T = (A,R)\), \(T_1 = (A_1, R_1)\) and \(T_2 = (A_2, R_2)\)
be three enlargements. Assume we are given morphisms \(T_1 \to T\) and
\(T_2 \to T\) that induces open immersions on special fibers. Then
\begin{equation*}
(A_1 \otimes^{w}_A A_2, R_1 \otimes_R R_2)
\end{equation*}
has a natural structure of an absolutely affine enlargement and represents the fiber product
\(T_1 \times_T T_2\) in the category of absolutely affine enlargements.
\end{corollary}

\begin{proof}
Let \(\widehat{B}\) be the \(\varpi\)-adic completion of a weakly
complete, finitely generated algebra \(B\). Let \(S = R_1 \otimes_R R_2\). Then
\(T\), \(T_1\) and \(T_2\) naturally define enlargements with special fiber
\(\mathrm{Spec}(S)\), by restricting to the corresponding Zariski localization.
By Lemma~\ref{lemma:flatness-tensor-product}, the completed tensor product
\(\widehat{A}_1 \widehat{\otimes}_{\widehat{A}} \widehat{A}_2\) is flat over
\(V\). As the weakly completed tensor product is a subring of the completed
tensor product, the weakly completed tensor product is \(\varpi\)-torsion free,
i.e., flat, as well.
\end{proof}

\begin{definition}
\label{definition:site}
Let \(\text{Enl}(X/V)\) be the category whose objects are absolutely affine
enlargements \((A,R)\). By Corollary \ref{corollary:fiber-product-enlargement},
\(\text{Enl}(X/V)\) admits fiber products. We say a collection of morphisms of
enlargements \(\{T_i \to T\}_{i\in I}\) is a \emph{covering} if the special fibers
\(T_i \otimes_V k\) form an open covering of the special fiber of \(T\).
Endowing with this notion of covering \(\text{Enl}(X/V)\) becomes a category
with a pretopology. We use \((X/V)_{\text{MW}}\) to denote its sheaf topos, and
call it the \emph{Monsky--Washnitzer topos} of \(X/V\).
\end{definition}
\begin{instance}
\label{example:sheaves-on-mw}
We give a few examples of sheaves on the Monsky--Washnitzer topos of a
\(k\)-variety \(X\).

(i) The presheaf which assigns to each (absolutely affine)
enlargement \(T = (A \to R \leftarrow \mathcal{O}_X)\) the weakly complete,
finitely generated algebra \(A\) is a sheaf. It is denoted by
\(\mathcal{O}_{X/V}\).

(ii) The presheaf which assigns to \(T\) the dagger algebar
\(A[1/\varpi]\) is also a sheaf, which is denoted by
\(\mathcal{O}_{X/V}^{\text{an}}\). We define the \emph{Monsky--Washnitzer cohomology} of
\(X/K\) to be the sheaf cohomology of \(\mathcal{O}_{X/V}^{\text{an}}\):
\begin{equation*}
\mathrm{H}_{\mathrm{MW}}^{\bullet}(X/K) = \mathrm{H}^{\bullet}((X/V)_{\mathrm{MW}},\mathcal{O}_{X/V}^{\text{an}})
\end{equation*}
Note that since \((X/V)_{\text{MW}}\) is not noetherian, the Monsky--Washnitzer
cohomology is different from the sheaf cohomology of \(\mathcal{O}_{X/V}\)
tensored with \(K\). For smooth affine varieties over \(k\), we shall see below
that the present definition agrees with the classical one.

(iii) Let \(P\) be a widening. Then \(P\) defines a sheaf on the
Monsky--Washnitzer site which sends an absolutely affine enlargement \(T\) to the
set \(\mathrm{Hom}(T,P)\).
\end{instance}

The role of twisted coefficients in a Monsky--Washnitzer topos are played by the
so-called Monsky--Washnitzer isocrystals, as defined below.

\begin{definition}[Crystalline sheaves]
\label{definition:crystalline-sheaves}
A \emph{crystalline} sheaf \(\mathcal{F}\) of \(\mathcal{O}_{X/V}\)-modules (or
\(\mathcal{O}^{\text{an}}_{X/V}\)-modules) is an \(\mathcal{O}_{X/V}\)-module
(or an \(\mathcal{O}^{\text{an}}_{X/V}\)-module) in the topos
\((X/V)_{\textup{MW}}\) such that for any morphism \(g: T' \to T\) of
enlargements, the natural morphism
\begin{equation*}
g^* \mathcal{F}_{T} \to \mathcal{F}_{T'}
\end{equation*}
of Zariski sheaves is an isomorphism.

We say a sheaf \(\mathcal{F}\) of \(\mathcal{O}_{X/V}\)-modules
(resp.~\(\mathcal{O}^{\text{an}}_{X/V}\)-modules) is \emph{coherent}, if for each enlargement
\(T\), the Zariski sheaf \(\mathcal{F}_T\) is coherent
(resp.~isocoherent). Coherent crystalline \(\mathcal{O}_{X/V}\)-modules
are called \emph{Monsky--Washnitzer crystals}, coherent crystalline
\(\mathcal{O}^{\text{an}}_{X/V}\)-modules are called \emph{Monsky--Washnitzer
isocrystals}.
\end{definition}
\begin{instance}
\label{example:smooth-frame}
Let \(P\) be a weak formal scheme over \(V\) whose completion is smooth over
\(V\). Let \(X\) be a closed subvariety of \(P\otimes_V k\). Then
\((P,X,\mathrm{Id})\) is a widening.
\end{instance}

\begin{lemma}
\label{lemma:affine-frame-covering-terminal}
Let the notation be as in Example~\ref{example:smooth-frame}.
Let \((T,Z,z)\) be an absolutely affine enlargement,
then there is a morphism of widenings \(T \to P\).
\end{lemma}

\begin{proof}
Let \(\widehat{T}\) be the completion of \(T\). Then \(Z\) is a scheme of
definition of \(\widehat{T}\). Since \(\widehat{P}\) is formally smooth, there
exists a lift \(\widehat{u}: \widehat{T} \to \widehat{P}\) that reduces to
\(z: Z \to P\). By the Artin approximation theorem of weakly complete, finitely
generated algebras
\cite[(2.4.2)]{van-der-put:monsky-washnitzer-cohomology}
we get a morphism \(u: T \to P\) that approximates
\(\widehat{u}\). In particular we can arrange \(u\) to agree with
\(\widehat{u}\) on the special fiber.
\end{proof}

The following is a translation of the above result to abstract nonsense.

\begin{corollary}
\label{corollary:smooth-frame-cover-terminal-object}
Let the notation be as above. Let \(h_P\) be the sheaf represented by the widening
\(P\) in the Monsky--Washnitzer topos.
Then object \(h_P\) is a covering to the terminal object.
\end{corollary}

\begin{situation}
\label{situation:many-pieces}
For a \(k\)-variety \(X\) that is not a closed subscheme of a weak formal scheme
whose completion is smooth, we can use a variant of the above construction to
get a covering of the terminal object of the Monsky--Washnitzer topos, as
follows. Let \(\{U_i\}\) be a covering of \(X\) by affine open subschemes. Let
\(P_i\) be a weak affine space containing \(U_i\) as a closed subscheme of its
special fiber. Then \((P_i, U_i)\) is a widening for \(X/V\). Now view each
\((P_i,U_i)\) as a sheaf on the Monsky--Washnitzer site, then the collection of
sheaves \(\{(P_i,U_i)\}\) is a covering of the terminal object of the
Monksy-Washnitzer topos. The proof is precisely as above (but easier, as we can
just use the universal property of the integral Monsky--Washnitzer algebra, and
need not to resort to the Artin approximation).
\end{situation}
We close this section by proving the functoriality of Monsky--Washnitzer topoi.
Note that if \(f: X \to X'\) is a morphism of \(k\)-varieties. Then \(f\) in
general does not pull back enlargements. So we don't always have continuous
functors on the level of categories with pretopologies. Nevertheless \(f\) does
induce a morphism of topoi
\(f_{\textup{MW}}:(X/V)_{\textup{MW}}\to(X'/V)_{\textup{MW}}\) since we
have a naturally defined ``cocontinuous functor''. Let's review this notion first.

\begin{definition}
\label{definition-cocontinuous}
Let \(\mathcal{C}\) and \(\mathcal{D}\) be categories with pretopologies. Let
\(u : \mathcal{C} \to \mathcal{D}\) be a functor. The functor \(u\) is called
\emph{cocontinuous} if for every \(U \in \mathrm{Ob}(\mathcal{C})\) and every covering
\(\{V_j \to u(U)\}_{j \in J}\) of \(\mathcal{D}\) there exists a covering
\(\{U_i \to U\}_{i\in I}\) of \(\mathcal{C}\) such that the family of maps
\(\{u(U_i) \to u(U)\}_{i \in I}\) refines the covering
\(\{V_j \to u(U)\}_{j \in J}\).
\end{definition}

The following lemma tells us how the notion of cocontinuous functors are related
to morphisms of sheaf topoi. Recall that if \(u: \mathcal{C} \to \mathcal{D}\)
is a functor of categories, then
\begin{equation*}
u^p :
\textit{PSh}(\mathcal{D})
\longrightarrow
\textit{PSh}(\mathcal{C})
\end{equation*}
is the functor that associates to \(\mathcal{G}\) on \(\mathcal{D}\) the
presheaf \(u^p\mathcal{G} = \mathcal{G} \circ u\). This functor has both left
and right adjoints which we denote by \(u_{p}\) and \({}_{p}u\) respectively.
For their definitions, see \cite{stacks-project}, \href{http://stacks.math.columbia.edu/tag/00XF}{Tag \texttt{00XF}}.

\begin{lemma}
\label{lemma-cocontinuous-morphism-topoi}
Let \(\mathcal{C}\) and \(\mathcal{D}\) be categories with pretopologies. Let
\(u : \mathcal{C} \to \mathcal{D}\) be cocontinuous. Then the \(g_\ast = {}_{p}u\)
takes sheaves to sheaves. The pair \(g_\ast\) and \(g^{-1} = (u^p)^{\#}\) (\(\#\)
stands for the sheafification) define a morphism of topoi \(g\) from
\(\textit{Shv}(\mathcal{C})\) to \(\text{Shv}(\mathcal{D})\).
\end{lemma}

\begin{proof}
See \cite{stacks-project}, \href{http://stacks.math.columbia.edu/tag/00XO}{Tag \texttt{00XO}}.
\end{proof}

\begin{lemma}
\label{lemma:morphism-induces-morphism-mw-topoi}
Let \(f: X \to X'\) be a morphism of \(k\)-varieties. Then it induces a
canonical morphism of the Monsky--Washnitzer topoi:
\(f_{\textup{MW}}:(X/V)_{\textup{MW}}\to(X'/V)_{\textup{MW}}\).
\end{lemma}

\begin{proof}
Let \(S = (\mathfrak{Z},Z,z)\) be an enlargement for \(X/V\). Then
\((\mathfrak{Z},Z,f\circ z)\) is an enlargement for \(X'/V\). This defines a
functor \(u: \mathrm{Enl}(X/V) \to \mathrm{Enl}(X'/V)\). We shall verify that
\(u\) is cocontinuous in the sense of
Definition~\ref{definition-cocontinuous}. Let
\(\{T_i = (\mathfrak{V}_i, V_i, v_i)\}\) be a collection of enlargements of
\(X'\) that covers \(u(S)\). Since we are using Zariski topology, this means
that \(\{\mathfrak{V}_i\}\) is a Zariski open covering for \(\mathfrak{Z}\) and
\(\{V_i\}\) is a Zariski open covering of \(Z\). The morphism \(v_i\) thus
factors as
\begin{equation*}
V_i \hookrightarrow Z \xrightarrow{z} X \xrightarrow{f} X'.
\end{equation*}
In other words, \(T_i\) are naturally enlargements for \(X/V\).
Therefore the defining property of a cocontinuous functor is satisfied for
\(u\), even without refining the covering. This completes the proof.
\end{proof}

\section{Universal enlargements}
\label{sec:org525a0c6}

In this section we define the crucial tool in the of the study of the
Monsky--Washnitzer topos, the so-called \emph{universal enlargement} construction.
In the context of convergent topos this is due to Ogus. In the rigid-analytic
context this is Berthelot's tubular neighborhood construction.
\begin{situation}
\textbf{Admissible weak formal blowup.} Let \(A\) be an admissible (i.e., flat) weakly
complete, finitely generated algebra. Let \(X = \mathrm{Spwf}(A)\). Let \(I\) be
an ideal (necessarily coherent) of \(A\). Denote \(A_n = A/\varpi^{n+1}A\), and
\(I_n = \frac{I}{\varpi^{n+1}A \cap I}\). Since \(A\) has \(\varpi\)-adic
topology, \(I\) is \emph{open} if and only if it contains some power of \(\varpi\).
Assume that \(I\) is open, then for each \(n\), define
\begin{equation*}
X_{I_{n}} = \mathrm{Bl}_{I_n}(\mathrm{Spec}(A_n)) = \text{Proj}_X\left({\textstyle \bigoplus_{j=0}^{\infty} I_n^{j}}\right).
\end{equation*}
There is a natural map \(X_{I_n} \to \mathrm{Spec}(A_n)\), and by passing to
limits, a morphism of formal schemes
\begin{equation*}
\pi: \widehat{X}_{I} \to \widehat{X}
\end{equation*}
where \(\widehat{X}_{I} = \text{colim}\;X_{I_n}\) is called the \emph{admissible formal blowup}
of \(\widehat{X} = \mathrm{Spf}(\widehat{A})\) along \(I\). If instead of
performing completion we perform the weak completion, we get a weak formal
scheme \(X_{I}\).

Now let \(X\) be an admissible weak formal scheme with a covering of admissible
affine weak formal schemes \(\{U_{i} = \mathrm{Spwf}(A_i)\}\). Let \(\mathcal{A}\) be
an open coherent sheaf of ideals on \(X\). Then \(\mathcal{A}|_{U_i}\) is an
open coherent \(\mathfrak{a}_i\) of \(A_i\). Then we can define the admissible
weak formal blowups \(U_{i,\mathfrak{a}_i}\). It can be shown admissible formal
blowups satisfy a certain universal property similar to that of the usual
blowup, so we can glue the formal schemes \(U_{i,\mathfrak{a}_i}\) together, and
obtain another weak formal scheme \(X_{\mathcal{A}}\). This is called the \emph{admissible weak
formal blowup} of \(X\) along \(\mathcal{A}\).
\end{situation}

\begin{theorem}
\label{theorem:flat-admissible-formal-blow-up}
Let \(X\) be an admissible weak formal \(\mathfrak{o}\)-scheme. Let \(\mathcal{A}\)
be a sheaf of coherent open ideal. Then \(X_{\mathcal{A}}\) is an admissible
weak formal scheme over \(V\).
\end{theorem}

\begin{proof}
By \cite[(8.2/8)]{bosch:formal-rigid}, the completion of \(X\) is
\(\varpi\)-torsion free over \(V\). Therefore \(\mathcal{O}_{X}\), being a
subalgebra of \(\mathcal{O}_{\widehat{X}}\), is also \(\varpi\)-torsion free.
Since \(V\) is a valuation ring, being \(\varpi\)-torsion free is the same as
being flat.
\end{proof}

\begin{definition}
\label{definition:tube-of-closed-radii}
Let \(X\) be an admissible weak formal scheme. Let \(\mathcal{I}\) be a coherent open
ideal of \(\mathcal{O}_X\). In the sequel, let \(T_{m,\mathcal{I}}(X)\) be the
open formal subscheme of the admissible blowup
\(X_{\varpi+\mathcal{I}^{m}}\)
(which is the weak completion of the scheme
\(\text{Proj}_{X}\big(\bigoplus_{j=0}^{\infty} \{(\varpi)+\mathcal{I}^{m}\}^{j}\big)\))
defined by the nonvanishing locus of the
Cartier divisor \(\varpi\). Thus, the weak formal scheme \(T_{m,\mathcal{I}}(X)\) is
admissible, thanks to Theorem~\ref{theorem:flat-admissible-formal-blow-up}, and there exists a natural morphism of
weak formal schemes over \(\mathfrak{o}\):
\begin{equation*}
\tau_{m} : T_{m,\mathcal{I}}(X) \to X
\end{equation*}
such that the ideal
\((\varpi + \mathcal{I}^{m})\cdot \mathcal{O}_{T_{m,\mathcal{I}}(X)}\)
generated by the inverse image of \(\varpi\) and elements in \(\mathcal{I}^{m}\)
is the principal ideal \(\varpi \cdot \mathcal{O}_{T_{m,\mathcal{I}}(X)}\).
\end{definition}

\begin{instance}
\label{example:tube-of-a-point}
Let \(A = V\langle x \rangle^{\dagger}\).
Then \(A\) is an admissible weakly complete, finitely generated \(V\)-algebra.
Let \(X = \mathrm{Spwf}(A)\), and \(\mathcal{I} = (x)\). Then
\begin{equation*}
T_{n,\mathcal{I}}(X) = \mathrm{Spwf}\left(V\langle x,s \rangle^{\dagger}/(x^{n} - \varpi s) \right).
\end{equation*}
The generic fiber
\(T_{n,\mathcal{I}}(X)_{K}\) of \(T_{n,\mathcal{I}}(X)\) is the closed
overconvergent disk \(\mathrm{Sp}(K\langle x \rangle^{\dagger})\) in \(K\) of
radius \(|\varpi|^{1/n}\). When \(n \to \infty\), the
union of these closed (overconvergent) disks is the open disk of radius 1.
\end{instance}

\begin{lemma}
\label{lemma:tube-ind-system}
Let the notation be as in Definition~\ref{definition:tube-of-closed-radii}.
Suppose \(n_1 > n_2\) are non-negative integers.
Then there exists a unique commutative diagram
\begin{equation*}
\xymatrix{
T_{n_2,\mathcal{I}}(X) \ar[rd]_{\tau_{n_2}} \ar[rr]^{\tau_{n_1,n_2}} & & T_{n_1,\mathcal{I}}(X) \ar[ld]^{\tau_{n_1}} \\
& X &
}.
\end{equation*}
Thus \(\{T_{n,\mathcal{I}}(X)\}\) is naturally an ind-weak formal scheme.
\end{lemma}

\begin{proof}
The arrows exist even before we pass to the weak formal completion, as shown by
the following.
By the universal property of blowup, it suffices to prove the ideal in
\(\mathcal{O}_{T_{n_2,\mathcal{I}}(X)}\) generated by
\(\tau_{n_2}^{-1} \left((\varpi) + I^{n_1}\right)\) is the principal ideal \((\varpi)\).
It is clear that \(\varpi\) is in the ideal generated by
\(\left((p) + I^{n_1}\right)\). Let \(f_{1} \cdot f_{n_1}\) be local sections opf
\(\mathcal{I}^{n_1}\). On (the uncompleted version of)
\(T_{n_2,\mathcal{I}}(X)\), the open piece of
a certain \(\text{Proj}\) construction obtained by inverting a degree 1 section
\(\varpi\) of \(\mathcal{O}(1)\), the element
\begin{equation*}
g = \frac{f_{i_1}\cdots f_{i_{n_2}}}{\varpi}
\end{equation*}
is a well-defined local \emph{function} on \(T_{n_2,\mathcal{I}}(X)\).
Since \(n_1 > n_2\), we can write
\begin{equation*}
f_{1} \cdots f_{n_1} = g \cdot f_{n_2+1} \cdots f_{n_1} \cdot \varpi
\end{equation*}
on an open formal subscheme of the un-completed version of
\(T_{n_2,\mathcal{I}}(X)\). Clearly the
expression falls in the ideal generated by \(\varpi\). We win.
\end{proof}

\begin{lemma}
\label{lemma:rigid-tube-inclusion-are-weierstrass-domains}
Let the notation be as in Definition~\ref{definition:tube-of-closed-radii}.
Assume that \(X\) is an affine weak formal scheme.
Suppose \(n_1 > n_2\) are non-negative integers.
Then the morphism on generic fibers
\begin{equation*}
\tau_{n_1,n_2,K}:T_{n_2,\mathcal{I}}(X)_{K} \to T_{n_1,\mathcal{I}}(X)_{K}
\end{equation*}
(induced by \(\tau_{n_1,n_2}\) defined in Lemma~\ref{lemma:tube-ind-system})
is an admissible open immersion of affinoid dagger spaces represented by a
Weierstraß subdomain.
\end{lemma}

\begin{proof}
Clearly, if \(U_1 \subset U_2\) are two affinoid subdomains of an affinoid dagger space
\(U\), and if \(U_i\) are all Weierstraß domains of \(U\), then \(U_1\) is a
Weierstraß domain of \(U_2\). Therefore, it suffices to prove that the morphism
\begin{equation*}
\tau_{n,K} : T_{n,\mathcal{I}}(X)_{K} \to X_{K}
\end{equation*}
is an open immersion of a Weierstraß subdomain. But then according to the
definition, the former is obtained by forcing the affinoid functions \(g/\varpi\) on
\(X_{K}\) to be \(\leq 1\), where \(g \in\mathcal{I}^{n}\),
hence is indeed a Weierstraß domain. We win.
\end{proof}

\begin{lemma}
\label{lemma:density-inclusion}
Let hypotheses and the notation be as in Lemma~\ref{lemma:rigid-tube-inclusion-are-weierstrass-domains}.
Then the morphism \(\tau_{n_1,n_2, K}^{\ast}\) on affinoid algebras has
dense image (with respect to the affinoid topology).
\end{lemma}

\begin{proof}
This is a corollary of Lemma~\ref{lemma:rigid-tube-inclusion-are-weierstrass-domains}, because that whenever a morphism
\(\varphi: U' \to U\) is an open immersion of affinoid dagger spaces that is a
Weierstraß domain, \(\varphi^{\ast}\) always has dense image: if \(B\) is a
dagger algebra and \(h \in B\), then we can always approximate any element
in \(B\langle \zeta \rangle^{\dagger}/(\zeta - h)\) by elements of the form
\(\sum_{j=1}^{N} b_j h^{j}\), which lies in \(B\).
\end{proof}

\begin{lemma}
\label{lemma:union-generic-fiber-tubular-neighborhood}
Let the notation be as in Definition~\ref{definition:tube-of-closed-radii}.
We have
\begin{equation*}
\bigcup_{m=1}^{\infty} T_{m,\mathcal{I}}(X)_K = (Z)_{X}
\end{equation*}
where \(Z\) is the vanishing scheme of \(\mathcal{I}\).
The above covering is also an admissible covering.
\end{lemma}

\begin{proof}
This is a local problem. Then we can write down a presentation of
\(\mathcal{I}\), and use the explicit description of tubular neighborhood as
well as the generic fibers of \(T_{n,\mathcal{I}}\). In fact, after
inverting \(\varpi\), the \(m\)th universal enlargement becomes the domain
defined, in \(X_K\), by \(|f| \leq |\varpi|\), \(f\) runs through elements in
\(\mathcal{I}^{m}\), and \((Z)_X\) is defined by \(|g| < 1\), for \(g \in
\mathcal{I}\). The assertion is now clear.
\end{proof}
\begin{situation}
\label{situation:colim-topos}
Let the notation be as in Definition~\ref{definition:tube-of-closed-radii}. Let \(Z\) be the vanishing scheme of
\(Z\).
Write \(T_n = T_{n,\mathcal{I}}(X)\). Recall by definition we have natural
morphisms \(\gamma_n: T_n \to P\). We have set \(Z_n = \gamma_n^{-1}(Z)\), and
we have \((T_n \otimes_V k)^{\text{red}} = Z_n^{\text{red}}\). So the ambient
topological space of  \(T_n\)  is no different from that of \(Z_n\), which maps
to the ambient space of \(Z\) in a continuous fashion. So we get a morphism of
sheaf topoi: \(\gamma_n:T_n^{\sim} \to Z_{\text{Zar}}^{\sim}\). This is not a
morphism of ringed topoi though.

We define a site \(\vec{T}\) as follows.
The objects are the open subsets of some \(T_n\).
If \(U_n\) is open in \(T_n\) and \(U_m\) is open in \(T_m\), with \(n \leq m\),
then a morphism \(U_n \to U_m\) is a commutative diagram
\begin{equation*}
\xymatrix{
  U_n \ar[r] \ar[d] & U_m\ar[d] \\
  T_n \ar[r]        & T_m
}.
\end{equation*}
If \(n > m\) there is no morphism from \(U_n\) to \(U_m\). The coverings of
\(U_n\) are usual Zariski coverings.

We can similarly define a site \(\vec{T}_K\) for the generic fibers \(T_{n,K}\).
The specialization morphisms give rise to a morphism of sites
\(\vec{T}_K \to \vec{T}\) which is denoted by \(\vec{\mathrm{sp}}\).
Since \(T_{n,K}\) form an admissible nested covering of \((Z)_X\), there is a
natural full and faithful morphism \((Z)^{\sim}_X \to \vec{T}{}^{\sim}_K\) of
sheaf topoi, given by restrictions. Putting everything together, we get a commutative diagram of topoi:
\begin{equation}
\label{eq:theorem-b-diagram}
\xymatrix{
  \vec{T}{}^{\sim}_K \ar[d]_{\vec{\text{sp}}} & \ar[l]_{\text{res}} (Z)_{X}^{\sim} \ar[d]^{\text{sp}}\ar[r] & X_K \ar[d]^{\text{sp}} \\
  \vec{T}^{\sim}     \ar[r]^{\gamma}    & Z_{\text{Zar}}^{\sim} \ar[r]                   & X_{\text{Zar}}^{\sim}
}
\end{equation}
where \(\gamma_{\ast}(F_n)\) is the inverse limit of the sheaves
\(\gamma_{n\ast}(F_n)\).
If \((F_n)\) is a collection of coherent sheaves on \(\vec{T}_{K}\), then
\(\mathrm{R}^{i}\mathrm{sp}_{\ast}(F_n)=0\), since cohomology of coherent
sheaves on an affinoid dagger space is
zero~\cite[Proposition~3.1]{grosse-klonne:rigid-analytic-spaces-with-overconvergent-structure-sheaf}
for all \(i > 0\).
By definition, \(R^i\mathrm{res}_{\ast}(F)\) is zero for any sheaf on \((Z)_X\).
\end{situation}

\begin{definition}
\label{definition:admissible}
Let the notation be as in~\ref{situation:colim-topos}.
Let \(\mathcal{O}_{\vec{T}_K}=(\mathcal{O}_{T_{n,K}})\) be the sheaf on
\(\vec{T}_K\) defined by \(\mathrm{res}_{\ast}(\mathcal{O}_{(Z)_X})\). We say a
sheaf \((F_n)\) of \(\mathcal{O}_{\vec{T}_K}\)-modules is \emph{good} if
\begin{enumerate}
\item \(F_n\) is coherent over \(\mathcal{O}_{T_{n,K}}\),
\item Zariski locally on \(Z\), the transition maps
\(F_n \otimes \mathcal{O}_{T_{n,K}} \to F_{n-1}\)
are isomorphism.
\end{enumerate}
Then good sheaves are essential images of coherent
\(\mathcal{O}_{(Z)_X}\)-modules under the functor
\(res_{\ast}\). We say a sheaf on \(\vec{T}\) is good if it is an essential
image of a good sheaf on \(\vec{T}_K\) under \(\mathrm{sp}\).
\end{definition}
We next need an incarnation of Kiehl's Theorem B for dagger spaces.
Note that Theorem B for dagger spaces is significantly harder than the original.
The simple-minded method of reducing to completion does not work, as one does
not have control of the radii of the fringe algebras appeared in the process.
Luckily, the theorem is made available by
F.~Bambozzi~\cite{bambozzi:theorem-b-dagger-quasi-stein} for
quasi-Stein dagger spaces that are closed subspaces of products of dagger disks and
open disks.

\begin{lemma}
\label{lemma:good-sheaf-gamma-acyclic}
Let \((F_n)\) be a good sheaf on \(\vec{T}\). Then
\begin{enumerate}
\item \(R\gamma_{\ast}(F_n) = \gamma_{\ast}(F_n)\).
\item For any coherent sheaf \(\mathcal{F}\) on the dagger space \((Z)_X\), we have
\(R\mathrm{sp}_{\ast}(\mathcal{F}) = \mathrm{sp}(\mathcal{F})\).
\end{enumerate}
\end{lemma}

\begin{proof}
Since the problem is local, we are reduced to the situation where \(Z\) and
\(X\) are affine.

We first show that the first assertion follows from the second. Each \(F_n\) is
of the form \(G_n \otimes K\) for some coherent sheaf. Since \(T_n\) has Zariski
topology, we have \(\mathrm{H}^{e}(T_n, F_n) = \mathrm{H}^{e}(T_n,G_n)\otimes K\)
The latter vanishes for \(e > 0\). This shows that the functor
\(\vec{\text{sp}}\) is exact on good sheaves, or
\(\vec{\text{sp}}_{\ast} = R\vec{\text{sp}}_{\ast}\). Since
\(R\gamma_{\ast} \circ R\vec{\text{sp}}_{\ast} \circ R\text{res}_{\ast} = R \text{sp}_{\ast}\)
the vanishing of \(R^{i}\text{sp}\), \(R^{j}\vec{\text{sp}}\) and
\(R^{\ell}\text{res}_{\ast}\) (\(i,j,\ell>0\)) imply the vanishing of \(R^{s}\gamma_{\ast}\)
for all \(s > 0\) by the spectral sequence of composition of functors.

Now we turn to the proof of (2). With the hypothesis that \(Z\) and \(X\) are
affine, \((Z)_X\) is a closed subspace of a quasi-Stein
space satisfying the hypotheses
of~\cite[Corollary~4.22]{bambozzi:theorem-b-dagger-quasi-stein}.
In fact, \(X_K\) is a closed subspace of a dagger closed disk defined by a
collection of function, \((Z)_X\) is the closed subspace of
\(X_K \times D(0;1^{-})^r\) defined by the equations of \(t_i = f_i\), where
\(f_1,\ldots,f_r\) are a collection of generators of the ideal defining \(Z\).
Thereby by loc.~cit.~the cohomology groups
\(\mathrm{H}^{i}((Z)_X,\mathcal{F})\) are zero, i.e., ``Theorem B holds''.
The exactness of \(\mathrm{res}_{\ast}\), the vanishing of coherent cohomology
of affinod dagger spaces  (whence the vanishing of \(R^{i}\mathrm{sp}_{\ast}\)
for the specialization morphisms \(\mathrm{sp}: T_{n,K}\to T_{n}\)), and the
commutativity of the diagram~\eqref{eq:theorem-b-diagram} together imply the
lemma.
\end{proof}

\begin{definition}
\label{definition:analyti-cohomology}
Given a pair \((X,Z)\) as in~\ref{situation:colim-topos}, we define the
\emph{analytic cohomology} of \(Z\) with respect to \(X\) to be the de~Rham
cohomology \(\mathrm{H}_{\text{dR}}^{\bullet}((Z)_X)\). To be precise, on any
dagger space \(Y\) there is a complex of coherent sheaves
\(\Omega_{Y}^{\bullet}\) defined as usual by using exterior powers of the \emph{weakly
completed} 1-differentials on \(Y\), and the de~Rham cohomology of \(Y\) is
then the hypercohomology of this complex.
(We have avoided the cumbersome dagger notation, as henceforth we
consider only weakly completed and completer differentials. We shall put hats on
completed ones.)
\end{definition}
\begin{definition}
\label{definition:universal-enlargement}
Let \(X\) now be a variety over
\(k\). Let \((P,Z,z)\) be a widening of \(X\). Then the construction
in~\ref{situation:colim-topos} gives rise to a collection of enlargements
\begin{equation*}
(T_{n,I_Z}(X),Z_n,z_n : Z_n \to Z \to X).
\end{equation*}
The ind-object \(\{(T_{n,Z}(X),Z_n,z_n : Z_n \to Z \to X)\}\) in the
Monsky--Washnitzer site is called the universal enlargements associated to the
widening \(P\).
\end{definition}

\begin{situation}
\label{situation:same-sheaf}
It is not hard to see that the colimit of
\(\{(T_{n,Z}(X),Z_n,z_n : Z_n \to Z \to X)\}\) (in the Monsky--Washnitzer topos)
surjects onto the sheaf represented by the widening \(P\). To check that is an
isomorphism, it suffices to show the transition maps are injective as sheaves.
Checking locally, this is implied by the density result of the Weierstraß
domains. Thereby the colimit of
\(\{(T_{n,Z}(X),Z_n,z_n : Z_n \to Z \to X)\}\) (in the Monsky--Washnitzer topos)
is isomorphic to the sheaf represented by \(P\).
\end{situation}

\section{Computing Monsky--Washnitzer cohomology}
\label{sec:org0ac363b}
\noindent
\textbf{Overview.}
In this section we prove that we can use analytic cohomology to compute
Monsky--Washnitzer cohomology. To begin with, we draw a big diagram of topoi
indicating what will be going on. In the sequel, let \(X\) be a variety over
\(k\). Let \(P\) be a smooth weak formal scheme containing \(X\) as a
subvariety. Therefore \((P,X,\mathrm{Id})\) is a widening, which also defines a
sheaf on the Monsky--Washnitzer site. We use \((X/V)_{\text{MW}}|_P\) to denote
the relative site consisting of enlargements of \((X/V)\) together with a
specified arrow (of sheaves) to the widening \(P\).
\begin{equation*}
\xymatrix{
  (X/V)_{\text{MW}}|^{\sim}_P \ar[d]_{j_P} \ar[r]^{\varphi_{\ast}} & \vec{P}^{\sim}   \ar[d]^{\gamma}& \vec{P}_K^{\sim} \ar[l]_{\vec{\text{sp}}}\\
  (X/V)^{\sim}_{\text{MW}}  \ar[r]_{u_{X/V}}                       &  X^{\sim} &  \ar[l]^{\text{sp}} (X)_P^{\sim} \ar[u]_{\text{res}}
}
\end{equation*}
We need the whole section to explain the details of these functors, but let us
at this point have an overview of the entire process.
The left diagram is about analytic cohomology.
The analytic cohomology is computed by
\(R\Gamma(X,R\mathrm{sp}_{\ast}\Omega_{(X)_P}^{\bullet})\).
By the ``Theorem B'' of Bambozzi--Kiehl,
Lemma~\ref{lemma:good-sheaf-gamma-acyclic}, the specialization
functor is acyclic with respect to coherent sheaves, so we can suppress the
derived symbol and use the following detour to compute the analytic cohomology:
\begin{equation*}
R\Gamma(X,\mathrm{sp}_{\ast}\Omega_{(X)_P}^{\bullet}) = R\Gamma(X, (\gamma_{\ast} \circ \vec{\mathrm{sp}}_\ast \circ \mathrm{res}_{\ast})\Omega^{\bullet}_{(X)_P}).
\end{equation*}
We have also explained that the all three functors \(\vec{\text{sp}}\),
\(\mathrm{res}\), and \(\gamma\) are of vanishing higher direct images in
Lemma~\ref{lemma:good-sheaf-gamma-acyclic}.

On the left square, the functor \(\varphi_{\ast}\) is not part of a morphism of
topoi. But it is exact. On the Monsky--Washnitzer topos, we can
construct, using differential forms, a complex of sheaves which we call
\(\mathrm{dR}_P\). It turns out
\(Ru_{X/V,\ast}\mathrm{dR}_P^{\bullet} = \mathrm{sp}_{\ast}\Omega_{(X)_P}^{\bullet}\).
So the analytic cohomology can also be computed by
\begin{equation*}
\mathrm{R}\Gamma(X,Ru_{X/V,\ast} \mathrm{dR}_P^{\bullet}).
\end{equation*}
It turns out \(Rj_{P,\ast} = j_{P,\ast}\) for the above complex,
and there is a so-called ``Poincaré lemma'', i.e., there is a quasi-isomorphism
\begin{equation*}
\mathcal{O}_{X/V}^{\text{an}} \to \mathrm{dR}_{P}^{\bullet}
\end{equation*}
Thereby we see the above complex can be replaced by
\begin{equation*}
R\Gamma(X,Ru_{X/V,\ast} \mathcal{O}_{X/V}^{\text{an}}),
\end{equation*}
that is, \emph{the Monsky--Washnitzer cohomology agrees with the analytic cohomology}.

Let us carry out the above scheme of proof.
\begin{situation}
\label{situation:notation-frame}
In this section, let \(P\) be a quasi-compact smooth weak formal scheme over
\(V\). Let \(X\) be a closed subvariety of \(P \otimes_V k\). Then
\((P,X,\mathrm{Id}_X)\) is a widening for \(X/V\). When there is no ambiguity we
shall also use \(P\) to denote the widening \((P,X,\mathrm{Id}_X)\). This is
less accurate but more economical.

Let \(\{P_n\}\) be the universal enlargement associated to the widening \(P\).
See Definition~\ref{definition:universal-enlargement}. Let \(\vec{P}\) be
the topos defined in~\ref{situation:colim-topos}.
\end{situation}

\begin{situation}
\label{situation:relative-topos}
\textbf{Definitions.}
(i) The relative site \((X/V)_{\text{MW}}|_P\) is the site whose objects are arrows
\(T \to P\) of widenings, in which \(T\) is an enlargement of \((X/V)\). The
notion of coverings is defined in the obvious fashion. Let
\((X/V)_{\text{MW}}|_P^{\sim}\) be its sheaf topos. There is a natural morphism of
topoi \(j_P : (X/V)_{\text{MW}}|_P^{\sim} \to (X/V)_{\text{MW}}^{\sim}\).

(ii) We want to define a morphism of topoi \(\varphi_P\) from
\((X/V)_{\text{MW}}|_P^{\sim}\) to \(P_{\text{Zar}}^{\sim}\).
Let \(\mathcal{F}\) be a Zariski sheaf on \(P\), \(\mathcal{G}\) a sheaf on
\((X/V)_{\text{MW}}|_P\). Let \(U\) be an open weak formal subscheme of \(P\).
Then \((U,U\cap X, U\cap X \to X)\) is a widening which represents an object in the
relative topos that is denoted by \(h_U\). Let \(\{U_n\}\) be its associated
universal enlargement of \(U\), so \(h_U = \text{colim}\, h_{U_n}\).
Let \(g: T \to P\) be a morphism of widenings, in which \(T\) is an enlargement.
So \(g\) is an object in the relative site.
Define
\begin{itemize}
\item \(\varphi_P^{-1}(\mathcal{F})(g) = g^{-1}(\mathcal{F})(T)\), and
\item \(\varphi_{P,\ast}(\mathcal{G})(U) = \mathrm{Hom}(h_U, \mathcal{G}) = \lim_n
  \mathcal{G}(U_n)\).
\end{itemize}
If \(\mathcal{F}\) is a coherent sheaf on \(P\), define
\[
\varphi_P^{\ast}\mathcal{F} =
\varphi_P^{-1}\mathcal{F}\otimes_{\varphi_P^{-1}\mathcal{O}_P}
j_{P}^{-1}\mathcal{O}_{X/V}^{\text{an}}.
\]
Clearly \(\varphi_{P\ast}\varphi_P^{\ast}\mathcal{F}=\mathcal{F}\otimes_{V}K\).

(iii) Let \(\mathcal{G} \in (X/V)_{\text{MW}}|_P^{\sim}\).
We define the functor, which is not a part of a morphism of topoi, \(\varphi_{\ast}\)
to be the functor sending \(\mathcal{G}\) to \((\mathcal{G}|_{P_n})_n\), that
is, if \(U_n\) is an open of \(P_n\), then \((U_n, U_n \cap (P_n \otimes_V k))\) is
naturally an enlargement of \((X/V)\) lying over \(P\), and
\(\mathcal{G}|_{P_n}(U_n)\) is the
value of \(\mathcal{G}\) at \((U_n, (U_n \cap P_n \otimes_V k))\).
\end{situation}

\begin{lemma}
\label{lemma:exact-phi}
The functor \(\varphi_{\ast}\) is exact.
\end{lemma}

\begin{proof}
This follows from the definition of covering in the Monsky--Washnitzer site. In
fact, if we are given an exact sequence of abelian sheaves
\(\mathcal{F}^{\bullet}\), and a cocycle of the complex
\(\mathcal{F}^{\bullet}(U_n)\), then there exists an open covering of the
enlargement \(U_n\) such that the restriction of this cocycle becomes coboundary
we restricting to the smaller opens. Since covering are defined using Zariski
open covering, this exactness implies the exactness of the restriction complex
\(\varphi_{\ast}(\mathcal{F}^{\bullet})\).
\end{proof}
Since \(\varphi_{\ast}\) is not part of a morphism of topoi, we need the
following lemma, whose proof is identical
to~\cite[Lemma~0.4.1]{ogus:convergent-topos}, that connects the
cohomology of a sheaf \(\mathcal{F}\) on \((X/V)_{\text{MW}}|_P\) and the
cohomology of \(\varphi_{\ast}\mathcal{F}\) on the site \(\vec{P}\).

\begin{lemma}
\label{lemma:ogus-flasque}
The functor \(\varphi_{\ast}\) takes injective abelian sheaves to flasque abelian
sheaves. Thereby for any abelian sheaf \(\mathcal{F}\), we have
\(\mathrm{H}^i((X/V)_{\text{MW}}|_P,\mathcal{F})=\mathrm{H}^{i}(\vec{P},\varphi_{\ast}\mathcal{F})\).
\end{lemma}
\begin{situation}
\label{situation:how-to-compute-jp-direct-image}
Next we study the functor \(j_{P,\ast}\) and its derived functors.
In this paragraph we first discuss how to compute the direct image of \(j_P\).
Let \((T,i: Z \to T,z:Z \to X)\) be an absolutely affine enlargement. Let
\(\mathcal{G} =\varphi_{P}^{\ast} \mathrm{sp}_{\ast}\mathcal{F}\) for some coherent sheaf
\(\mathcal{F}\) on the weak formal scheme \(P\). Then by definition
\(j_{P\ast}(\mathcal{G}) = \mathcal{G}(P \times T)\). Recall, \(P\times T\) is
the widening whose weak formal scheme is \(P \times_V T\), which is again a weak
formal scheme; the central part is \(Z \times_{z,X} X = Z\). So we view \(Z\) as
a closed subscheme of \(P \times_V T\) via the product embedding \(z \times i\).

However, since \(P \times_V T\) is not an enlargement, we cannot ``evaluate'' a
sheaf on it. What we could do is to take the Hom group
\begin{equation*}
\mathrm{Hom}(P \times T, \mathcal{G})
\end{equation*}
of sheaves, i.e., we regard \(P \times_V T\) as a sheaf. On the other hand, we
know \(P \times_V T\) is the colimit of its universal enlargement, i.e., as
sheaves we have
\begin{equation*}
P \times T = \text{colim}_{n} (P \times T)_n.
\end{equation*}
As \((P \times T)_n\) is an enlargement lying over \(P\), we then conclude that
\begin{equation*}
\mathrm{Hom}(P \times T, \mathcal{G}) = \lim \mathcal{G}(P \times T)_n = \lim \mathcal{F}(P\times T)_{n,K} = \mathcal{F}((Z)_{P\times_V T}).
\end{equation*}
From this we conclude that the direct image functor \(j_{P\ast}\) is computed by
means of the \emph{value of the sheaf on the tubular neighborhood}.
\end{situation}

\begin{lemma}
\label{lemma:acyclic-jp-and-u}
Let \(\mathcal{F}\) be a coherent \(\mathcal{O}_{P}\)-sheaf.
Then \(\mathcal{G} = \varphi_P^{\ast}\mathcal{F}\) is \(j_{P}\)-acyclic, and
\(j_{P\ast}(\mathcal{G})\) is \(u_{X/V}\)-acyclic.
Cf.~\cite[0.4.3]{ogus:convergent-topos}.
\end{lemma}

\begin{proof}
We have a commutative diagram
\begin{equation*}
\xymatrix{
  (X/V)_{\text{MW}}|_{P}^{\sim} \ar[r]^{\varphi_{\ast}}  \ar@/^2pc/[rr]_{\varphi_P} \ar[d]_{j_P} & \vec{P}{}^{\sim}  \ar[d]_{\gamma} \ar[r]^{\tau} & P^{\sim} \\
  (X/V)_{\text{MW}}^{\sim} \ar[r]_u  & X^{\sim}  \ar[ru] &
}
\end{equation*}
of morphisms of topoi (except \(\varphi_{\ast})\). Let \(\mathcal{F}\) be a
coherent sheaf on \(P\). Then using definition one checks that
\(\tau^{\ast}(\mathcal{F} \otimes K) = \varphi_{\ast}\mathcal{G}\).
By ``Theorem B'', Lemma~\ref{lemma:good-sheaf-gamma-acyclic},
\(R^{i}\gamma_{\ast}(\tau^{\ast}(\mathcal{F} \otimes K)) = 0\)  for \(i > 0\).
Therefore we have
\begin{align*}
  Ru_{\ast} Rj_{P\ast} \mathcal{G} &= R\gamma_{\ast}R\varphi_{\ast}\mathcal{G} \\
                                   &= \gamma_{\ast}\varphi_{\ast}\mathcal{G} \\
                                   &= u_{\ast}j_{P\ast} \mathcal{G}.
\end{align*}
The acyclicity of \(j_{P}\mathcal{G}\) with respect to \(u_{X/V}\) then follows
immediately from the spectral sequence of composition of functors and the
acyclicity of \(\mathcal{G}\) with respect to \(j_{P}\).

Let us now prove the vanishing of \(R^{m}j_{P\ast}\mathcal{G}\) for \(m > 0\).
To this end, it suffices to prove that for any absolutely affine enlargement
\((T,Z,z)\), the Zariski sheaf \(R^{m}j_{P\ast}\mathcal{G}_{T}\) is zero.
Consider the projection
\begin{equation*}
\rho = \mathrm{pr}_2|_{(Z)_{P\times T}}:  (Z)_{P\times T} \to T_K.
\end{equation*}
Then the Zariski sheaf \(R^{m}j_{P\ast}\mathcal{G}_{T}\) equals
\begin{equation*}
R^{m}\rho_{\ast}\mathcal{F}_K.
\end{equation*}
This can be seen by chasing a diagram of product widenings.
The above sheaf vanishes for \(m > 0\) since the projection is locally isomorphic to
\(D(0;1^{-})^N \times T_K \to T_K\) as dagger spaces, and the open polydisk has
no cohomology by applying ``Theorem B'',
Lemma~\ref{lemma:good-sheaf-gamma-acyclic}.
\end{proof}
\begin{situation}
\label{situation:de-rham-complex}
Let \(\Omega^{i}_{P}\) be
\(\varphi_{P}^{\ast}\Omega_{P/V}^{i,\dagger}\)
(see~\ref{situation:relative-topos}),
i.e., the ring-theoretic pullback of the weakly completed de~Rham complex
of the weak formal scheme \(P\).

On the weak formal scheme \(P\), the exterior differentials
\(d: \Omega_{P/V}^{i,\dagger}\otimes K \to \Omega_{P/V}^{i+1,\dagger} \otimes K\)
are only \(K\)-linear, not \(\mathcal{O}_{P}\otimes K\)-linear.
Therefore the sheaves \(\Omega_P^{i}\) do not form a complex in an obvious
fashion. Our first mission is to define natural exterior differentials
\(d: \Omega_{P}^{i} \to \Omega_{P}^{i+1}\), and get a natural
complex \(\Omega_{P}^{\bullet}\) on the Monsky--Washnitzer site.

To this end, we need to define the sheaf relative differential forms for
a morphism of dagger spaces. This process occupies
\ref{definition:dagger-module} --- \ref{lemma:relative-differential-product}.
In this course, although not strictly needed, we use the notion of bornological
spaces to make the treatment conceptual, but in the end bornology does not
appear in our result. Bambozzi's
thesis~\cite{bambozzi:generalization-affinoid-varieties},  which contains a
complete index for terminologies and definitions, is our reference for
bornological spaces and bornological algebras.
\end{situation}

Let \(A\) be an affinoid dagger \(K\)-algebra. In the sequel, we fix a
presentation \(A = W_{m}/I\) where \(I \subset T_n(\rho_0)\).

\begin{definition}
\label{definition:dagger-module}
A \emph{dagger} \(A\)-module is a bornological \(A\)-module \(M\) such that there
exists a strict epimorphism \(A^{n} \to M\) for some finite cardinal \(n\).
\end{definition}

Here is an Artin-Rees type lemma for dagger modules asserting that any finitely
generated \(A\)-module has a canonical bornological module that is compatible
with submodules.

\begin{lemma}
\label{lemma:artin-rees-for-dagger}
If \(A\) is a \(K\)-dagger algebra, then each finitely generated \(A\)-module
admits a unique bornology making it into a dagger \(A\)-module. Moreover, with
respect to this canonical bornology, the category of finitely generated
\(A\)-modules is equivalent to the category of dagger \(A\)-modules.
\end{lemma}

\begin{proof}
See Corollaries 6.15, 6.16 of F.~Bambozzi's
thesis~\cite{bambozzi:generalization-affinoid-varieties}.
\end{proof}

\begin{situation}
\label{situation:relative-dagger-algebra}
Let \(B\) be a dagger \(A\)-algebra. This means that \(B\) is of the form
\(W_{m+n}/\{IW_{m+n} + J\}\). One particular example of \(B\) is
\(A\langle t_1,\ldots,t_n\rangle^{\dagger} = W_{m+n}/IW_{m+n}\).
Let us arrange such that \(I, J\) are all
convergent on the polydisk of radius \(\rho_0\). Thus we can present \(B\) as a
filtered colimit \(B = \mathrm{colim}(B_\rho)\) with injective transition maps.
Here \(B_{\rho} = \mathfrak{T}_{m+n}(\rho)/\{I\mathfrak{T}_{m+n}(\rho)+J\}\).
\(B\) is complete as a bornological \(k\)-space, as it is a colimit of complete
Banach spaces; but it may not be complete as a ``multiplicatively convex bornological
algebra'' in the sense of Bambozzi (which is defined as the direct limit of
Banach \(k\)-algebras with monomorphic transition maps), as it is could contain
torsion elements.

Let \(B \otimes_A B\) be the completed tensor product obtained as the colimit of
the completed tensor product \(B_{\rho} \otimes_{A_{\rho}} B_{\rho}\) with the
colimit bornology. Then we see \(B\), \(B \widehat{\otimes}_A B\) are all dagger algebras
over \(k\) (by writing down a presentation of \(B\)).
\end{situation}

\begin{situation}
\label{situation:bounded-differentials}
\textbf{Construction of \(\Omega_{B/A}^{1}\).}
Let \(M\) be a dagger \(B\)-module. A \emph{bounded} \(A\)-linear derivation is a
derivation \(D: B \to M\) that is also bounded with respect to the bornology of
\(B\) inherited from the presentation and the canonical bornology of \(M\).

Let \(M\) be a bornological \(B\)-module. Consider the algebra \(B \ast M\)
whose ambient \(B\)-module is
\begin{equation*}
B \oplus M
\end{equation*}
with multiplication \((b + m)(b' + m') = bb' + bm' + b'm\). Then \(B \oplus M\)
is a bornological \(B\)-module with the product bornology. The multiplication of
\(B \ast M\) is bounded, since the action of \(B\) on \(M\) is bounded.
The group of bounded \(A\)-derivations is identified with the subset of bounded
\(A\)-algebra homomorphism from \(B\) into \(B \ast M\) that reduces to the
identity modulo \(M\).

The multiplication defines a bounded map \(\mu: B \otimes_A B \to B\) of
bornological \(A\)-algebras. The maps \(B \to B \otimes_A B\) defined by
\(d_0(b) = b \otimes 1\) and \(d_1(b) = 1 \otimes b\) are bounded. Define
\(I = \mathrm{Ker}(\mu)\). Let \(B^{(i)} = (B \otimes_A B)/I^{i+1}\).
Clearly the bounded homomorphisms \(d_0\) and \(d_1\) are liftings of the
identity \(B = B^{(0)}\). Thus \(\mathrm{d} = d_1 - d_0\) is a derivation valued
in the ideal \(I\). Since \(I\) is finitely generated, \(\mathrm{d}\) is
necessarily bounded.

Given a bounded \(A\)-linear derivation \(D: B \to M\), where \(M\) is a
bornological \(B\)-algebra. We get two bounded maps \(B \to B \ast M\) given by
\(\mathrm{Id}_B + 0\) and \(\mathrm{Id}_B + D\). This gives a bounded
homomorphism, by the universal product of the completed tensor product
\begin{equation*}
B \widehat{\otimes}_A B \to B \ast M
\end{equation*}
Note that \(I\) is mapped to \(M\) since the composition
\begin{equation*}
B \widehat{\otimes}_A B \to B \ast M  \to B
\end{equation*}
is none other than the multiplication map.
Since \(M^2 = 0\), \(I^2\) is sent to zero, and we obtain a homomorphism of
\(B\)-modules \(I/I^2 \to M\). This proves that \(\mathrm{d}: B \to I/I^2\)
is the universal bounded derivation. Since we do not use usual Kähler
differentials for dagger algebras, we will denote \(I/I^2\) by
\(\Omega_{B/A}^{1}\) without ambiguity.
\end{situation}

\begin{lemma}
\label{lemma:differential-relative-disk}
Let \(A\) be a dagger \(k\)-algebra.
Let \(B = A \langle t_1, \ldots,t_n \rangle^{\dagger}\).
Then \(\Omega_{B/A}^{1} = \bigoplus_{i=1}^{n} B \cdot \mathrm{d}t_i\).
\end{lemma}

\begin{proof}
For any \(1 < \rho < \rho_0\), define
\(B_{\rho} = A_{\rho}\langle t/\rho\rangle\).
The maps \(d_{i,\rho} : B_{\rho}\to B_{\rho}^{(1)}\) are defined and when
\(\rho < \rho'\), we have \(d_{i,\rho'}|_{B_{\rho}} = d_{i,\rho}\). This means
that the \(A_{\rho}\)-modules \(\Omega_{B_{\rho}/A_{\rho}}^1\), the module of
completed Kähler differential for \(B_{\rho}/A_{\rho}\) form a bornology
of \(\Omega_{B/A}^{1}\). Clearly this agrees with the canonical bornology as
\(\Omega_{B/A}^{1}\) is finitely generated over \(B\) by definition. The
assertion now follows from the standard result for completed differentials, as
the \(\mathrm{d}t_i\) form a basis all \(\Omega^1_{B_{\rho}/A_{\rho}}\)
\end{proof}

The module completed differentials clearly sheafifies (as it is defined by
universal property). For a dagger space \(X\) over \(\mathrm{Sp}(A)\), we can
thus define \(\Omega_{X/A}^{1}\). One can then develop the theory of smooth
and étale morphism of dagger spaces. But for our purpose we only need to know
the following two simple facts.

\begin{lemma}
\label{lemma:relative-open-disk}
Let \(A\) be an affinoid dagger \(K\)-algebra.
Let \(X = \bigcup_{0 <\eta < 1}\mathrm{Sp}(A\langle t_1/\eta, \ldots, t_n/\eta \rangle^{\dagger})\).
Then \(\Omega^{1}_{X/A}\) is the free \(\mathcal{O}_{X}\)-module generated by
\(\mathrm{d}t_i\).
\end{lemma}

\begin{proof}
The differential forms \(\mathrm{d}t_i\) are a compatible system of
bases of the sheaf of relative differentials on the admissible covering
\(\mathrm{Sp}(A\langle t/\eta\rangle^{\dagger})\).
\end{proof}

\begin{lemma}
\label{lemma:relative-differential-product}
Let \(A\) be an affinoid dagger algebra.
Let \(X\) be a dagger space.
Form the product \(X_A = X \times \mathrm{Sp}(A)\).
Then \(\Omega^{1}_{X_A/A} = \mathrm{pr}_1^{\ast}\Omega^1_{A/K}\).
\end{lemma}

\begin{proof}
We can assume \(X = \mathrm{Sp}(B)\) is affinoid.
Let \(M\) be a bornological
\(A \widehat{\otimes}_K B\)-module. Then it could be viewed as a bornological
\(B\)-module in a natural way. Any \(A\)-linear bounded derivation from the
tensor product into \(M\) gives rise to a bounded \(K\)-linear derivation from
\(B\) into \(M\). Thereby we get a bounded \(B\)-linear homomorphism
\(\Omega_{B/K}^{1} \to M\) that lifts to a bounded \(A\)-linear homomorphism
\(A \otimes_K \Omega_{B/A}^1 \to M\), in a unique fashion.
\end{proof}
\begin{situation}
\label{situation:relative-de-rham-complex}
\textbf{Definition.}
Let \(A\) be an affinoid dagger algebra. Let \(f: X \to \mathrm{Sp}(A)\) be a
morphism of dagger spaces. Then the universal \(A\)-linear bounded derivation
\(\mathrm{d}: \mathcal{O}_{X} \to \Omega_{X/A}^{1}\) prolongs to a complex
\(\Omega_{X/A}^{\bullet}\), whose arrows \(\mathrm{d}\) are \(A\)-linear, and
\(\Omega_{X/A}^{i} = \bigwedge^{i}\Omega_{X/A}^{i}\). This complex is called the
\emph{relative (bounded) de~Rham complex}.
\end{situation}

\begin{situation}
\label{situation:construction-differential}
\textbf{Construction.}
Let \((T,i:Z \to T,z: Z \to X)\) be an absolutely affine enlargement.
Let \(T_K = \mathrm{Sp}(A)\) be the generic fiber of \(T\). Since \(T\) is an
enlargement, we have \((Z)_T = T_K\). We form the product
\begin{equation*}
(T \times_V P, Z \xrightarrow{i,\mathrm{z}} T \times_V P, z: Z \to X)
\end{equation*}
of widenings. Then according
to~\ref{situation:how-to-compute-jp-direct-image},
\(j_{P\ast}\Omega_{P}^{i}(T) = \mathrm{pr}_2^{\ast}\Omega_{P_K}^{i}((Z)_{T \times_V P})\)
where
\begin{equation*}
\mathrm{pr}_2: T_K \times P_K \to P_K
\end{equation*}
is the projection to the second factor. Since \((Z)_{T\times_V P}\) is an open
subspace of \(T_K \times P_K\), we have, by
Lemma~\ref{lemma:relative-differential-product}, that
\begin{equation*}
j_{P\ast}\Omega_{P}^{i}(T) = \Omega_{(Z)_{T\times_V P}/T_K}^{i}((Z)_{T\times_{V}P}).
\end{equation*}
We then can use the exterior differentials,
Definition~\ref{situation:relative-de-rham-complex}, to define
\begin{equation*}
\mathrm{d}: j_{P\ast}\Omega_P^{i}(T) \to j_{P\ast}\Omega_P^{i+1}(T).
\end{equation*}
Thereby we get a complex whose entries are \(j_{P\ast}\Omega_{P}^{i}\), with
differentials defined above. We denote this complex by
\(\mathrm{dR}_{P}^{\bullet}\).
\end{situation}

\begin{lemma}
\label{lemma:poincarelemma}
The natural map
\begin{equation*}
\mathcal{O}_{X/V}^{\textup{an}} \to \mathrm{dR}^{\bullet}_{P}
\end{equation*}
is a quasi-isomorphism of sheaves on the Monsky--Washnitzer topos.
\end{lemma}

\begin{proof}
By the weak fibration theorem, Lemma~\ref{lemma:weak-fibration-theorem}, and induction on relative dimension,
it suffices to prove that the relative
de~Rham cohomology of \(\mathrm{Sp}(A) \times D(0;1^{-})\) is
quasi-isomorphic to \(A\). This is a two term complex
\begin{equation*}
\lim\nolimits_{0<\eta<1} \mathrm{colim}_{1<\rho<r}\; A_{\rho}\langle t/\eta \rangle^{\dagger} \xrightarrow{\partial/\partial t}
\lim\nolimits_{0<\eta<1} \mathrm{colim}_{1<\rho<r}\; A_{\rho}\langle t/\eta \rangle^{\dagger}.
\end{equation*}
Set \(R_{\eta} = \mathrm{colim}_{1<\rho<r}\; A_{\rho}\langle t/\eta \rangle^{\dagger}\).
Pick an element \(f \in R_{\eta}\).
Then there exists \(\rho\), such that \(f \in A_{\rho}\langle t/\eta \rangle\)
(without dagger), standard integration gives rise to an element in
\(g \in A_{\rho}\langle t/\eta'\rangle\), where \(\eta'\) is a number slightly
bigger than \(\eta\), such that \(\partial g/\partial t=f\).
This means that the morphism \(\partial/\partial_t\) on pro-object
\(\{R_{\eta}\}\) is surjective. The kernel equals \(A\) since we can embed the
map in the ring of formal power series \(A[\![t]\!]\) and check the kernel
there.
\end{proof}

\begin{lemma}
\label{lemma:zariski-image-de-rham-complex}
We have an isomorphism of complexes of Zariski sheaves
\begin{equation*}
u_{X/V,\ast}(\mathrm{dR}_{P}^{\bullet}) \cong \mathrm{sp}_{\ast}\Omega_{(Z)_{P}}^{\bullet}.
\end{equation*}
\end{lemma}

\begin{proof}
Since \(u_{X/V,\ast}\) and \(j_{P\ast}\) preserves global sections, we have
\begin{equation*}
\Gamma(U, u_{X/V,\ast}j_{P\ast}\Omega_{P}^{i}) = \Omega^{i}_{(Z)_P}((Z)_P) =
\Gamma(U, \mathrm{sp}_{\ast}\Omega^{i}_{(Z)_P})
\end{equation*}
for any Zariski open subset of \(X\). Thus, the entries of the two complexes are
equal. It remains to identify differentials. Let \((T,Z,z)\) be an enlargement.
The morphism \(T \times P \to P\) is visualized by the morphism
Consider the morphism
\begin{equation*}
\xymatrix{
  (Z)_{P\times T} \ar[r] \ar[d] & (X)_{P} \ar[d] \\
  P_K \times T_K \ar[r]       & P_K
}
\end{equation*}
of dagger spaces. By Lemma~\ref{lemma:relative-differential-product}, the
global sections of the relative de~Rham complex on \((Z)_{P\times T}\) and
the global section of the absolute de~Rham complex on \((X)_P\) has a
natural chain map, and is functorial in \(T\). Thus we get a natural morphism
of chain complexes
\begin{equation*}
\Gamma(X,\mathrm{sp}_{\ast}\Omega_{(X)_P}^{\bullet}) \to \Gamma(X, u_{\ast}\mathrm{dR}_P).
\end{equation*}
which is isomorphism on entries. The lemma follows since we can replace \(X\) by
its Zariski opens.
\end{proof}

\begin{corollary}
\label{corollary:comparison}
Let \(X\) be a \(k\)-variety. Let \(P\) be an admissible weak formal scheme over
\(V\) whose completion is smooth. Assume that \(P \otimes_V k\) contains \(X\)
as a closed subscheme. Then the analytic cohomology of \(X\) with respect to
\(P\) is isomorphic to the Monsky--Washnitzer cohomology of \(X\).
\end{corollary}

\begin{proof}
By Lemma~\ref{lemma:acyclic-jp-and-u},
\(u_{X/V,\ast}\) is acyclic with respect to sheaves
\(j_{P\ast}\Omega_P^{i}\), thereby
\begin{equation*}
Ru_{X/V,\ast} \mathrm{dR}_{P} = u_{X/V,\ast}\mathrm{dR}_{P}.
\end{equation*}
By the Poincaré lemma~\ref{lemma:poincarelemma},
the Zariski cohomology of the left hand side equals the Monsky--Washnitzer
cohomology; by Lemma~\ref{lemma:good-sheaf-gamma-acyclic}  and
Lemma~\ref{lemma:zariski-image-de-rham-complex} the Zariski cohomology of
the right hand side  equals the analytic cohomology.
\end{proof}

\bibliographystyle{alpha}
\bibliography{/home/dzhang/personal/OK/newbib.bib}

\newcommand{\etalchar}[1]{$^{#1}$}
\begin{thebibliography}{{The}17}


\bibitem[Bam16a]{bambozzi:generalization-affinoid-varieties}
Federico Bambozzi.
\newblock On a generalization of affinoid varieties. Ph.D thesis.
\newblock 2016.

\bibitem[Bam16b]{bambozzi:theorem-b-dagger-quasi-stein}
Federico Bambozzi.
\newblock Theorems {A} and {B} for dagger quasi-{S}tein spaces. (To appear on  \emph{The Quarterly Journal of Mathematics}).
\newblock 2016.

\bibitem[Ber96]{berthelot:rigid-cohomology-compact-support}
Pierre Berthelot.
\newblock Cohomologie rigide et cohomologie rigide {\`a} supports propres
  (premi{\`e}re partie).
\newblock {\em pr{\'e}publication de l'IRMAR 96-03}, 1996.

\bibitem[Bos14]{bosch:formal-rigid}
Siegfried Bosch.
\newblock {\em Lectures on formal and rigid geometry}, volume 2105 of {\em
  Lecture Notes in Mathematics}.
\newblock Springer, Cham, 2014.

\bibitem[FGI{\etalchar{+}}05]{fantechi-et-al:fga-explained}
Barbara Fantechi, Lother G{\"o}ttsce, Luc Illusie, Steven~L. Kleiman, Nitin
  Nitsure, and Angelo Vistoli.
\newblock {\em Fundamental algebraic geometry: Grothendieck's {FGA} explained},
  volume 123 of {\em Mathematical Surveys and Monographs}.
\newblock American Mathematical Soc., 2005.

\bibitem[FK17]{fujiwara-kata:foundation-rigid-geometry}
Kazuhiro Fujiwara and Fumiharu Kato.
\newblock Foundations of rigid geometry. i.
\newblock 2017.

\bibitem[GK00]{grosse-klonne:rigid-analytic-spaces-with-overconvergent-structure-sheaf}
Elmar Grosse-Kl{\"o}nne.
\newblock Rigid analytic spaces with overconvergent structure sheaf.
\newblock {\em J. Reine Angew. Math.}, 519:73--95, 2000.

\bibitem[Gro68]{grothendieck:crystals-de-rham-cohomology}
A.~Grothendieck.
\newblock Crystals and the de {R}ham cohomology of schemes.
\newblock In {\em Dix expos\'{e}s sur la cohomologie des sch\'{e}mas}, volume~3
  of {\em Adv. Stud. Pure Math.}, pages 306--358. North-Holland, Amsterdam,
  1968.
\newblock Notes by I. Coates and O. Jussila.

\bibitem[LS11]{lestum:overconvergent}
Bernard Le~Stum.
\newblock The overconvergent site.
\newblock {\em M{\'e}m. Soc. Math. Fr. (N.S.)}, (127):vi+108 pp. (2012), 2011.

\bibitem[Mer72]{meredith:weak-formal-scheme}
David Meredith.
\newblock Weak formal schemes.
\newblock {\em Nagoya Math. J.}, 45:1--38, 1972.

\bibitem[MW68]{monsky-washnitzer:formal-cohomology-1}
Paul Monsky and Gerard Washnitzer.
\newblock Formal cohomology. {I}.
\newblock {\em Ann. of Math. (2)}, 88:181--217, 1968.

\bibitem[Ogu90]{ogus:convergent-topos}
Arthur Ogus.
\newblock The convergent topos in characteristic {\(p\)}.
\newblock In {\em The {G}rothendieck {F}estschrift, {V}ol.\ {III}}, volume~88
  of {\em Progr. Math.}, pages 133--162. Birkh{\"a}user Boston, Boston, MA,
  1990.

\bibitem[Shi02]{shiho:crystalline-fundamental-group-2}
Atsushi Shiho.
\newblock Crystalline fundamental groups. {II}. {L}og convergent cohomology and
  rigid cohomology.
\newblock {\em J. Math. Sci. Univ. Tokyo}, 9(1):1--163, 2002.

\bibitem[{The}17]{stacks-project}
{The Stacks Project Authors}.
\newblock {S}tacks {P}roject, 2009 -- 2017.

\bibitem[vdP86]{van-der-put:monsky-washnitzer-cohomology}
Marius van~der Put.
\newblock The cohomology of {M}onsky and {W}ashnitzer.
\newblock {\em M{\'e}m. Soc. Math. France (N.S.)}, (23):4, 33--59, 1986.
\newblock Introductions aux cohomologies \(p\)-adiques (Luminy, 1984).

\end{thebibliography}
\end{document}